\documentclass[11pt, notitlepage]{article}
\usepackage{amssymb,amsmath,comment}
\catcode`\@=11 \@addtoreset{equation}{section}
\def\thesection{\arabic{section}}

\def\theequation{\thesection.\arabic{equation}}
\catcode`\@=12
\usepackage{colortbl}
\usepackage[mathscr]{eucal}
\usepackage{epsf}
\usepackage{esint}
\usepackage{a4wide}

\newcommand{\noi} {\noindent}

\newcommand{\Tail} {\mathrm{Tail}}
\newcommand{\loc} {\mathrm{loc}}

\usepackage[all]{xy}
\catcode`\@=11
\def\theequation{\@arabic{\c@section}.\@arabic{\c@equation}}
\catcode`\@=12

\newtheorem{Theorem}{Theorem}[section]
\newtheorem{Lemma}[Theorem]{Lemma}

\newtheorem{Remark}[Theorem]{Remark}
\newtheorem{Definition}[Theorem]{Definition}

\begin{document}

{\vspace{0.01in}}

\title
{ Some local properties of subsolution and supersolutions for a doubly nonlinear nonlocal $p$-Laplace equation}

\title
{ \sc Some local properties of subsolutons and supersolutions for a doubly nonlinear nonlocal  parabolic $p$-Laplace equation}
\author{Agnid Banerjee, Prashanta Garain and Juha Kinnunen}

\maketitle

\begin{abstract}
{We establish a local boundedness estimate for weak subsolutions to a doubly nonlinear parabolic fractional $p$-Laplace equation. 
Our argument relies on energy estimates and a parabolic nonlocal version of De Giorgi's method. 
Furthermore, by means of a new algebraic inequality, we  show that positive weak supersolutions satisfy a reverse H\"older inequality. Finally, we also prove a logarithmic decay estimate for  positive supersolutions. }

\medskip

\noi {Key words: Doubly nonlinear parabolic equation, fractional $p$-Laplace equation,  energy estimates, De Giorgi's method}.

\medskip

\noi \textit{2010 Mathematics Subject Classification:} 35K92, 35B45, 35R11.
\end{abstract}

\section{Introduction}
This work studies the local behaviour of subsolutions and supersolutions to the doubly nonlinear parabolic nonlocal problem
\begin{equation}\label{maineqn}
\partial_t(u^{p-1})+\mathcal{L}u=0\text{ in }\Omega\times(0,T),
\quad p>2,
\end{equation}
where $\Omega\subset\mathbb{R}^n$ is a bounded smooth domain, $T>0$ and the operator $\mathcal{L}$ is defined by
$$
\mathcal{L}u(x,t)=\text{P.V.}\int_{\mathbb{R}^n}|u(x,t)-u(y,t)|^{p-2}(u(x,t)-u(y,t))K(x,y,t)\,dy,
$$
and where P.V. stands for  the principal value. We assume that $K$ is a symmetric kernel with respect to $x$ and $y$ satisfying 
\begin{equation}\label{kernel}
\frac{\Lambda^{-1}}{|x-y|^{n+ps}}\leq K(x,y,t)\leq\frac{\Lambda}{|x-y|^{n+ps}},
\end{equation}
uniformly in $t\in(0,T)$ for some $\Lambda\geq 1$ and $s\in(0,1)$. If $K(x,y,t)=|x-y|^{-(n+ps)}$, then $\mathcal{L}$ becomes the fractional $p$-Laplace operator $(-\Delta)_p^{s}$, 
which further reduces to the fractional Laplacian $(-\Delta)^s$ for $p=2$.

The partial differential equation in \eqref{maineqn} constitutes a nonlocal counterpart of the doubly nonlinear equation,
\begin{equation}\label{pLap}
\partial_t(u^{p-1})-\text{div}(|\nabla u|^{p-2}\nabla u)=0.
\end{equation} 
We refer the reader to \cite{Kinhigherint, Verenacontinuity, Gvespri, Kin-Kuusi, Kuusiholder2, Kuusiholder,  Liao} and the references therein. 
To the best of our knowledge, there is no literature available concerning the corresponding nonlocal equation.
This paper is a first step towards a regularity theory where we  prove a local boundedness estimate for weak subsolutions to  \eqref{maineqn} when $p>2$. 
To this end, we establish an energy estimate (Lemma \ref{Subenergyestimate}) and apply De Giorgi's method to obtain our main result (Theorem \ref{Localbound1}). 
We also prove a reverse H\"older inequality for strictly positive weak supersolutions (Theorem \ref{revholder}) by means of a new algebraic inequality (Lemma \ref{Inequality1}) and a logarithmic decay estimate (Lemma \ref{Logestimatelemma}). In particular, Lemma \ref{Inequality1} generalizes an inequality due to Felsinger and Kassmann for $p=2$, see Lemma 3.3 in \cite{Kassweakharnack}.  Finally we  note that in the local case as for \eqref{pLap}, such a reverse H\"older property as well as the logarithmic estimate  constitute some of the key ingredients in the proof of  weak Harnack inequality, see for instance \cite{Kin-Kuusi}. To the best of our knowledge, weak Harnack inequality seems to be an open question in the nonlocal case for the doubly nonlinear equation \eqref{maineqn} and therefore we believe that  our results will be important in investigating such question along with further qualitative and quantitative properties of weak solutions to \eqref{maineqn}. 

Fractional Laplace equations have been a topic of considerable attention recently. We refer to the survey \cite{Hitchhiker'sguide} by Di Nezza, Palatucci and Valdinoci for an elementary introduction to the theory of the fractional Sobolev spaces and fractional Laplace equations.
For globally nonnegative solutions of the elliptic fractional Laplace equation $(-\Delta)^s u=0$, Landkof \cite{Nonlocalharnackfirst} obtained scale invariant Harnack inequality, which fails for sign changing solutions as shown by Kassman \cite{KassmanHarnack}. Indeed, an additional tail term appears in the Harnack estimate. Castro, Kuusi and Palatucci studied local boundedness and H$\ddot{\text{o}}$lder continuity results for the equation $(-\Delta_p)^s u=0$ with $p>1$ in \cite{Kuusilocal}.
They also obtained Harnack inequality with a tail dealing with sign changing solutions in \cite{Kuusiharnack}. The nonhomogeneous case $(-\Delta_p)^s u=f$ has been settled for local and global boundedness along with a discussion of eigenvalue problem by Brasco and Parini \cite{Braspar}. Moreover in this case, Brasco, Lindgren and Schikorra established higher and optimal regularity results in \cite{Braslinsck}. See also \cite{Brascolind, Cozzi} and the references therein.    

In the parabolic setting, for the fractional heat equation, 
$
\partial_t u+(-\Delta)^s u=0,
$
weak Harnack inequality has been established by Felsinger and Kassman in \cite{Kassweakharnack}, see also \cite{Kassmanchaker, KassmanSch} for related results. Caffarelli, Chan and Vasseur established boundedness and H$\ddot{\text{o}}$lder continuity results in \cite{Cafchanvas} for different type of kernels. For regularity results up to the boundary, see  \cite{Xavier}. Bonforte, Sire and V\'{a}zquez established optimal existence and uniqueness results in \cite{Vazquez}, along with a scale invariant Harnack inequality for globally positive solutions. For sign changing solutions, Str\"{o}mqvist proved Harnack inequality with a tail in \cite{Martinharnack}, see \cite{Kim} for a different approach. 

In the nonlinear framework, we mention the work of V\'{a}zquez \cite{Vazquez1} where global boundedness results for the equation 
$$
\partial_t u+(-\Delta_p)^s u=0
$$
have been obtained. See also \cite{Rossi}. For such an equation, local boundedness result with a tail term has been  investigated by Str\"{o}mqvist in \cite{Martinlocal}. More recently, H$\ddot{\text{o}}$lder continuity results has been established for the same equation by Brasco, Lindgren and Str\"{o}mqvist in \cite{Martincont}. 
In the doubly nonlinear case, Hynd and Lindgren \cite{LinHynd} addressed the question of pointwise behavior of viscosity solutions for the following doubly nonlinear equation
\[
|\partial_t u|^{p-2}\partial_t u+(-\Delta)_{p}^s u=0.
\]
See also \cite{LinHynd1, LinHynd2} for related results in the local case. 

This paper is organized as follows: In Section 2, we introduce some basic notations, gather some preliminary results that are relevant to our work and then state our main results. In Section 3-5, we prove our main results. Finally, in Section 6, the appendix, we give a proof of the algebraic inequality in Lemma \ref{Inequality1} which is applied in the proof of Theorem \ref{revholder}. 
\section{Preliminaries and main results}
We  first present some facts about fractional Sobolev spaces. For more details we refer the reader to \cite{Hitchhiker'sguide}.
\begin{Definition}
Let $1<p<\infty$ and $0<s<1$ and assume that $\Omega\subset\mathbb{R}^n$ is an open and connected subset of $\mathbb R^n$. 
The fractional Sobolev space $W^{s,p}(\Omega)$ is defined by
$$
W^{s,p}(\Omega)=\Big\{u\in L^p(\Omega):\frac{|u(x)-u(y)|}{|x-y|^{\frac{n}{p}+s}}\in L^p(\Omega\times \Omega)\Big\}
$$
and endowed with the norm
$$
\|u\|_{W^{s,p}(\Omega)}=\Big(\int_{\Omega}|u(x)|^p\,dx+\int_{\Omega}\int_{\Omega}\frac{|u(x)-u(y)|^p}{|x-y|^{n+sp}}\,dx\,dy\Big)^\frac{1}{p}.
$$
{The fractional Sobolev space with zero boundary values is defined by}
$$
W_{0}^{s,p}(\Omega)={\big\{u\in W^{s,p}(\mathbb{R}^n):u=0\text{ on }\mathbb{R}^n\setminus\Omega\big\}}.
$$
\end{Definition}

Both $W^{s,p}(\Omega)$ and $W_{0}^{s,p}(\Omega)$ are reflexive Banach spaces, see \cite{Hitchhiker'sguide}. The parabolic Sobolev space $L^p(0,T;W^{s,p}(\Omega))$ is the set of measurable functions $u$ on $\Omega\times(0,T)$, $T>0$, such that 
$$
||u||_{L^p(0,T;W^{s,p}(\Omega))}=\Big(\int_{0}^{T} ||u(\cdot,t)||^p_{W^{s,p}(\Omega)}\,dt\Big)^\frac{1}{p}<\infty.
$$
The spaces $W^{s,p}_{\loc}(\Omega)$ and $L^p_{\loc}(0,T;W^{s,p}_{\loc}(\Omega))$ are defined analogously. Next we discuss Sobolev embedding theorems, see \cite{Hitchhiker'sguide}.
We write by $C$ to denote a positive constant which may vary from line to line or even in the same line depending on the situation. If $C$ depends on $r_1,r_2,\dots,r_k$, we write $C=C(r_1,r_2,\dots,r_k)$.

\begin{Theorem}\label{ellipticembedding1}
Let $1<p<\infty$ and $0<s<1$ with $sp<n$ and $\kappa^{*}=\frac{n}{n-sp}$. For every $u\in W^{s,p}(\mathbb{R}^n)$, we have
$$
\|u\|^p_{L^{\kappa^{*} p}(\mathbb{R}^n)}\leq\int_{\mathbb{R}^n}\int_{\mathbb{R}^n}\frac{|u(x)-u(y)|^p}{|x-y|^{n+sp}}\,dx\,dy.
$$
If $\Omega$ is  a bounded  extension domain for $W^{s,p}$ and $u\in W^{s,p}(\Omega)$, then for any $\kappa \in [1, \kappa^*]$, 
$$
\|u\|_{L^{\kappa p}(\Omega)}\leq C(\Omega)||u||_{W^{s,p}(\Omega)}.
$$
If $sp=n,$ then the above inequalities hold for any $\kappa \in[1,\infty)$. 
For $sp>n$, the second inequality holds for any $\kappa \in[1,\infty]$. 
\end{Theorem}

For the following Sobolev type inequality, we refer to \cite[Lemma 2.1]{Martinlocal}.
For $x_0\in\mathbb{R}^n$ and $r>0$, $B_r(x_0)=\{x\in\mathbb R^n:|x-x_0|<r\}$ denotes the ball in $\mathbb{R}^n$ of radius $r$ and center $x_0$. 
The barred integral sign denotes the corresponding integral average.

\begin{Lemma}\label{ellipticembedding2}
{Let $1<p<\infty$ and $0<s<1$.}
Assume that $u\in W^{s,p}(B_r)$, where $B_r=B_r(x_0)$, and let 
$\kappa^{*}=\frac{n}{n-sp}$, if $sp<n$, and $\kappa^{*}=2$, if $sp\geq n$.
There exists a constant $C=C(n,p,s)$ such that for every $\kappa\in[1,\kappa ^{*}]$, we have
$$
\Big(\fint_{B_r}|u(x)|^{\kappa p}\,dx\Big)^\frac{1}{\kappa}
\leq Cr^{sp-n}\int_{B_r}\int_{B_r}\frac{|u(x)-u(y)|^p}{|x-y|^{n+sp}}\,dx\,dy+C\fint_{B_r}|u(x)|^p\,dx.
$$
\end{Lemma}

Next we state the  parabolic Sobolev inequality as in \cite[Lemma 2.2]{Martinlocal}.

\begin{Lemma}\label{parabolic embedding}
{Let $p$, $s$ and} $\kappa^{*}$ be as in Lemma \ref{ellipticembedding2}. Assume that $u\in L^p(t_1,t_2;W^{s,p}(B_r))$. There exists a constant $C=C(n,p,s)$ such that for every $\kappa\in[1,\kappa^{*}]$, we have
\begin{gather*}
\int_{t_1}^{t_2}\fint_{B_r}|u(x,t)|^{\kappa p}\,dx\,dt
\leq C\Bigr(r^{sp-n}\int_{{t_1}}^{{t_2}}\int_{B_r}\int_{B_r}\frac{|u(x)-u(y)|^p}{|x-y|^{n+sp}}\,dx\,dy\,dt+\int_{t_1}^{t_2}\fint_{B_r}|u(x,t)|^p\,dx\,dt\Bigl)\\
\cdot\Bigl(\sup_{t_1<t<t_2}\fint_{B_r}|u(x,t)|^\frac{p\kappa^{*}(\kappa-1)}{\kappa^{*}-1}\,dx\Bigr)^\frac{\kappa^{*}-1}{\kappa^{*}}.
\end{gather*}
\end{Lemma}

We now state the following weighted Poincar\'{e} inequality in fractional Sobolev spaces, see \cite[Corollary 6]{Kassmaninequality}.

\begin{Lemma}\label{wgtPoincareunitball}
Let $1<p<\infty$, $0<s_0\leq s<1$. 
Assume that $\phi(x)=\Phi(|x|)$ is a radially decreasing function on $B_1=B_1(0)$. 
Then there exists a constant $C=C(p,n,s_0,\mathcal{\phi})$ such that for all $f\in L^p(B_1)$,
$$
\int_{B_1}|f(x)-f^{\phi}_{B_1}|^p \phi(x)\,dx\leq C(1-s)\int_{B_1}\int_{B_1}\frac{|f(x)-f(y)|^p}{|x-y|^{n+ps}}\min\{\phi(x),\phi(y)\}\,dx\,dy,
$$
where
$$
f_{B_1}^{\phi}=\frac{\int_{B_1}f(x)\phi(x)\,dx}{\int_{B_1}\phi(x)\,dx}.
$$ 
\end{Lemma}

Using change of variables in Lemma \ref{wgtPoincareunitball}, we obtain the following  weighted Poincar\'{e} inequality  which will be useful in establishing  a logarithmic estimate for weak supersolutions ( see Lemma \ref{Logestimatelemma}).

\begin{Lemma}\label{wgtPoincare}
Let $1<p<\infty$, $0<s<1$ and $\psi(x)=\Psi(|x-x_0|)$ be a radially decreasing function on $B_r=B_r(x_0)$. Then there exists a constant $C=C(n,p,s)$ such that for every $f\in L^p(B_r)$, 
$$
\int_{B_r}|f(x)-f^{\psi}_{B_r}|^{p}\psi(x)\,dx\leq Cr^{ps}\int_{B_r}\int_{B_r}\frac{|f(x)-f(y)|^p}{|x-y|^{n+ps}}\min\{\psi(x),\psi(y)\}\,dx\,dy,
$$
where
$$
{f^{\psi}_{B_r}=\frac{\int_{B_r}f(x)\psi(x)\,dx}{\int_{B_r}\psi(x)\,dx}}.
$$
\end{Lemma}

We also need the following real analysis lemmas. For the proof of Lemma \ref{Deg2} below, see  \cite[Lemma 4.1]{Dibe}.

\begin{Lemma}\label{Deg2}
Let $(Y_j)_{j=0}^{\infty}$ be a sequence of positive real numbers satisfying
$Y_{j+1}\leq c_0 b^{j} Y_j^{1+\beta}$,
for some constants $c_0>1$, $b>1$ and $\beta>0$. 
If $Y_0\leq c_{0}^{-\frac{1}{\beta}}b^{-\frac{1}{\beta^2}}$, then
$\lim_{j\to\infty}\,Y_j=0$.
\end{Lemma}

The next inequality is as in  \cite[Lemma 3.1]{Kuusilocal}.

\begin{Lemma}\label{epsilon}
Let $p\geq 1$ and $\epsilon\in(0,1]$. Then for every $a,b\in\mathbb{R}^n$, we have 
$$
|a|^p\leq |b|^p+C(p)\epsilon|b|^p+\big(1+C(p)\epsilon\big)\epsilon^{1-p}|a-b|^p,
$$
where $C(p)=(p-1)\Gamma(\max\{1,p-2\})$ and $\Gamma$ denotes the gamma function.
\end{Lemma}

The following elementary inequality will play a crucial role in the proof of  reverse H$\ddot{\text{o}}$lder inequality for supersolutions  as in Theorem \ref{revholder} below.
A proof for Lemma \ref{Inequality1} is given in the appendix. This generalizes an inequality of Felsinger and Kassmann \cite{Kassweakharnack} to the $p$-case.
\begin{Lemma}\label{Inequality1}
Let $a,b>0$, $\tau_1,\tau_2\geq 0$. Then for any $p>1$, there exists a constant $C=C(p)>1$ large enough such that
\begin{equation}\label{Ineeqn}
\begin{split}
&|b-a|^{p-2}(b-a)(\tau_1^{p}a^{-\epsilon}-\tau_2^{p}b^{-\epsilon})
\geq\frac{\zeta(\epsilon)}{C(p)}\Big|\tau_2 b^\frac{p-\epsilon-1}{p}-\tau_1 a^\frac{p-\epsilon-1}{p}\Big|^p\\
&\qquad-\Big(\zeta(\epsilon)+1+\frac{1}{\epsilon^{p-1}}\Big)\big|\tau_2-\tau_1\big|^p\big(b^{p-\epsilon-1}+a^{p-\epsilon-1}\big),
\end{split}
\end{equation}
where $0<\epsilon<p-1$ and $\zeta(\epsilon)=\epsilon(\frac{p}{p-\epsilon-1})^p$. If $0<p-\epsilon-1<1$, we may choose $\zeta(\epsilon)=\frac{\epsilon p^p}{p-\epsilon-1}$ in \eqref{Ineeqn}.
\end{Lemma}

For $v,k>0$, the auxiliary function defined by
$$
\xi((v-k)_{+})=\int_{k^{p-1}}^{v^{p-1}}\big(\eta^\frac{1}{p-1}-k\big)_{+}\,d\eta
=(p-1)\int_{k}^{v}(\eta-k)_{+}\eta^{p-2}\,d\eta,
$$
would be very useful to deduce the energy estimate below. Indeed, from \cite[Lemma 2.2]{Verenacontinuity}, we have the following result.
\begin{Lemma}\label{Auxfnlemma}
There exists a constant $\lambda=\lambda(p)>0$ such that for all $v,k>0$, we have
$$
\frac{1}{\lambda}(v+k)^{p-2}(v-k)_{+}^2\leq \xi((v-k)_{+})\leq\lambda(v+k)^{p-2}(v-k)_{+}^2.
$$
\end{Lemma}
For more applications of such functions in the doubly nonlinear context, we refer to \cite{Verenacontinuity, Gvespri, Kuusiholder}.

For $t_0\in(r^{sp},T-r^{sp})$, we consider the space-time cylinders
\[
U^{-}(r)=U^{-}(x_0,t_0,r)=B_r(x_{0})\times(t_0-r^{sp},t_0)
\]
and
\[
U^{+}(r)=U^{+}(x_0,t_0,r)=B_r(x_0)\times(t_0,t_0+r^{sp}).
\]
We denote the positive and negative parts of $u$ by 
\[
u_{+}(x,t)=\max\{u(x,t),0\}
\quad\text{and}\quad 
u_{-}(x,t)=\max\{-u(x,t),0\},
\]
respectively. For any $a,b\in\mathbb{R}$, we have $|a_{+}-b_{+}|\leq|a-b|$ which implies $u_{+}\in W^{s,p}(\Omega)$ when $u\in W^{s,p}(\Omega)$. 
Analogously, we have $u_{-}\in W^{s,p}(\Omega)$. 
Throughout the paper, we denote by
\[
\mathcal{A}(u(x,y,t))=|u(x,t)-u(y,t)|^{p-2}(u(x,t)-u(y,t))
\quad\text{and}\quad
d\mu=K(x,y,t)\,dx\,dy.
\]
It is well known that a tail term appears in nonlocal problems.
If $u$ is a measurable function in $\mathbb{R}^n\times(0,T)$ and $x_0\in\mathbb{R}^n$, $r>0$, $0<t_1<t_2<T$, the parabolic tail of $u$ with respect to $x_0$, $r$, $t_1$ and $t_2$ is defined by
\begin{equation}\label{taildef}
\Tail_{\infty}(u;x_0,r,t_1,t_2)=\Bigr(r^{sp}{\sup_{t_1<t<t_2}}\int_{\mathbb{R}^n\setminus B_r(x_0)}\frac{|u(x,t)|^{p-1}}{|x-x_0|^{n+sp}}\,dx\Bigl)^\frac{1}{p-1}.
\end{equation}

Next we define the notion  of weak sub and supersolution.
 
\begin{Definition}\label{subsupsolution}
A function $u\in L^\infty(0,T;L^{\infty}(\mathbb{R}^n))$, with $u>0\text{ in }\mathbb{R}^n\times(0,T)$, is a weak subsolution (or supersolution) of the equation \eqref{maineqn} if $u\in C_{\loc}(0,T;L^p_{\loc}(\Omega))\cap L^p_{\loc}(0,T;W_{\loc}^{s,p}(\Omega))$ and for every $\Omega' \times [t_1,t_2]\Subset \Omega \times (0,T)$, and  nonnegative test function $\phi\in W^{1,p}_{\loc}(0,T;L^p(\Omega'))\cap L^p_{\loc}(0,T;W_{0}^{s,p}(\Omega'))$, one has
\begin{equation}\label{weaksubsupsoln}
\begin{gathered}
\int_{\Omega'}u(x,t_2)^{p-1}\phi(x,t_2)\,dx
-\int_{\Omega'}u(x,t_1)^{p-1}\phi(x,t_1)\,dx
-\int_{t_1}^{t_2}\int_{\Omega'}u(x,t)^{p-1}\partial_t\phi(x,t)\,dx\,dt\\
+\int_{t_1}^{t_2}\int_{\mathbb{R}^n}\int_{\mathbb{R}^n}\mathcal{A}(u(x,y,t)){(\phi(x,t)-\phi(y,t))}\,d\mu\,dt\leq0\quad(\text{or $\geq0$})
\end{gathered}
\end{equation}
respectively.
\end{Definition}

\begin{Remark}\label{NeedofLinfinity}
The assumption $u\in L^\infty(0,T;L^\infty(\mathbb{R}^n))$ ensures that  the last term in the right hand side of \eqref{weaksubsupsoln} and the term $\Tail_{\infty}$ defined in  \eqref{taildef} are  finite. 
\end{Remark}

\begin{Remark}\label{Molifier}
The test functions in the energy estimates would depend on the solution $u$ itself and the use of the term $u_t$ can be justified by using the mollification in time defined for $f\in L^1(\Omega\times I)$ by
$$
f_h(x,t)=\frac{1}{h}\int_{0}^{t}e^{\frac{s-t}{h}}f(x,s)\,ds.
$$
For more details of $f_h$, we refer to \cite{Verenacontinuity, KLin}.
\end{Remark}

\subsection*{Statement of the main results.}
We now state our main results. Our first main result is following local boundedness estimate for subsolutions.
\begin{Theorem}\label{Localbound1}
Let $p>2$, $x_0\in\mathbb{R}^n$, $r>0$ and $t_0\in (r^{sp},T)$. Assume that $u$ is a weak subsolution of \eqref{maineqn} with
$
u>0\text{ in }\mathbb{R}^n\times(t_0-r^{ps},t_0).
$ 
Then {there exists a positive constant $C=C(n,p,s,\Lambda)$ such that} for any $\delta\in(0,1)$, we have
\begin{gather*}
\sup_{(x,t)\in U^{-}(\frac{r}{2})}u(x,t)
\leq C\delta^{-\frac{(p-1)\kappa n}{sp^2}}\Big(\fint_{U^{-}(r)}u(x,t)^p\,dx\,dt\Big)^\frac{1}{p}+\delta\,\Tail_{\infty}(u;x_0,\frac{r}{2},t_0-r^{sp},t_0),
\end{gather*}
where $\kappa=\frac{n+sp}{n}$, if $sp<n$, and $\kappa=\frac{3}{2}$, if $sp\geq n$.
\end{Theorem}

\begin{Remark}\label{Rncase}
One should note that in the case when  $\Omega=\mathbb{R}^n$,  for the validity of Theorem \ref{Localbound1} we only require  $u\in L^p((0,T);W^{s,p}(\mathbb{R}^n))\cap L^{\infty}((0,T);L^p(\mathbb{R}^n))$ which ensures that the $Tail_{\infty}$ is  finite  and thus one can avoid the  qualitative boundedness assumption on $u$ unlike that in the case of  bounded domain. \end{Remark}

Our second main result  constitutes the following reverse H\"older inequality  for positive supersolutions.

\begin{Theorem}\label{revholder}
Let $p>2$. Suppose that u is a weak supersolution of \eqref{maineqn} with $u\geq\rho>0$ in $\mathbb{R}^n\times(t_0,t_0+r^{ps})$, where $t_0\in(0,T-r^{ps})$. Then for any $\theta\in[\frac{1}{2},1)$ there exists positive constants $\mu=\mu(\kappa,p)$ and $C=C(n,p,q, s,\Lambda)\geq 1$ such that 
\begin{align}\label{rev}
\Big(\fint_{U^{+}(x_0,t_{0},\theta r)}{u(x,t)^{q}\,dx\,dt}\Big)^\frac{1}{q}
\leq\Big(\frac{C}{(1-\theta)^{\mu}}\fint_{U^{+}(x_0,t_0,r)}u(x,t)^{\bar{q}}\,dx\,dt\Big)^\frac{1}{\bar{q}},
\end{align}
for all $0<\bar{q}<q<q_0$ where $q_0=\kappa(p-1)$, provided $\kappa=\frac{n+sp}{n},$ if $sp<n$ and $\kappa=\frac{3}{2},$ if $sp\geq n$.
\end{Theorem}
\begin{Remark}
We would like to emphasize that the constant  $C$ in the  reverse H\"older inequality  \eqref{rev} above is independent of $\bar{q} $ as $\bar{q} \to 0$ and this is precisely where the algebraic lemma \ref{Inequality1} plays a crucial role. It is well known that such a  stable behaviour of the constant $C$ is needed in order  to establish  the Harnack inequality for local equations using the approach of Bombieri as in \cite{Bo} (see also \cite{Kin-Kuusi} for an adaptation of such an idea in the case of \eqref{pLap}). We therefore believe that such a reverse H\"older inequality will have similar future applications in the nonlocal case.
\end{Remark}

\section{Energy estimate}
To prove Theorem \ref{Localbound1}, we need  the following Caccioppoli type estimate for subsolutions.
\begin{Lemma}\label{Subenergyestimate}
Let $p>2$, $x_0\in\mathbb{R}^n$, $0<\tau_1<\tau_2$ and $l>0$ with $(\tau_1-l,\tau_2)\subset(0,T)$. 
Assume that $u$ is a weak subsolution of \eqref{maineqn} with
$u>0\text{ in }\mathbb{R}^n\times (\tau_1-l,\tau_2)$.
Let $k\in\mathbb{N}$ and denote $w(x,t)=(u-k)_{+}(x,t)$. Then there exists a positive constant {$C=C(n,p,s,\Lambda)$} such that 
\begin{align*}
&\int_{\tau_1-l}^{\tau_2}\int_{B_r}\int_{B_r}|w(x,t)\psi(x)-w(y,t)\psi(y)|^p \eta(t)^{p}\,d\mu\,dt
+C\sup_{\tau_1<t<\tau_2}\int_{B_r}w(x,t)^p\psi(x)^p\,dx\\&\leq\int_{\tau_1-l}^{\tau_2}\int_{B_r}\int_{B_r}|w(x,t)\psi(x)-w(y,t)\psi(y)|^p \eta(t)^{p}\,d\mu\,dt
+C\sup_{\tau_1<t<\tau_2}\int_{B_r}\xi(w)(x,t)\psi(x)^p\,dx\\
&\leq C\Bigg(\int_{\tau_1-l}^{\tau_2}\int_{B_r}\int_{B_r}{\max\{w(x,t),w(y,t)\}^p|\psi(x)-\psi(y)|^p}\eta(t)^p\,d\mu\,dt\\
&\qquad+\Big(\sup_{x\in\mathrm{supp}\,\psi,\,\tau_1-l<t<\tau_2}\int_{{\mathbb{R}^n\setminus B_r}}{\frac{w(y,t)^{p-1}}{|x-y|^{n+ps}}}\,dy\Big)
\int_{\tau_1-l}^{\tau_2}\int_{B_r}w(x,t)\psi(x)^p\eta(t)^p\,dx\,dt\\
&\qquad+\int_{\tau_1-l}^{\tau_2}\int_{B_r}\xi(w)\psi(x)^p\partial_t\eta(t)^p\,dx\,dt\Bigg),
\end{align*}
for all nonnegative $\psi\in C_{0}^{\infty}(B_r)$ and nonnegative $\eta\in C^{\infty}(\mathbb{R})$ such that $\eta(t)=0$ for $t\leq \tau_1-l$ and $\eta(t)=1$ for $t\geq \tau_1$.
\end{Lemma}

\begin{proof}
Since $p>2$, we observe that the first inequality  i.e.
\begin{align*}
&\int_{\tau_1-l}^{\tau_2}\int_{B_r}\int_{B_r}|w(x,t)\psi(x)-w(y,t)\psi(y)|^p \eta(t)^{p}\,d\mu\,dt
+C\sup_{\tau_1<t<\tau_2}\int_{B_r}w(x,t)^p\psi(x)^p\,dx\\&\leq\int_{\tau_1-l}^{\tau_2}\int_{B_r}\int_{B_r}|w(x,t)\psi(x)-w(y,t)\psi(y)|^p \eta(t)^{p}\,d\mu\,dt
+C\sup_{\tau_1<t<\tau_2}\int_{B_r}\xi(w)(x,t)\psi(x)^p\,dx
\end{align*}
follows directly from Lemma \ref{Auxfnlemma}. Therefore, it is enough to prove the second inequality. 

Let $t_1=\tau_1-l$ and $t_2=\tau_2$ and for fixed $t_1<l_1<l_2<t_2$ and $\epsilon>0$ small enough, following \cite{Verenacontinuity} we define the function $\zeta_{\epsilon}\in W^{1,\infty}\big((t_1,t_2),[0,1]\big)$ by
\[
  \zeta_{\epsilon}(t) :=
  \begin{cases}
    0 & \text{ for } t_1\leq t\leq l_1-\epsilon,\\
    1+\frac{t-l_1}{\epsilon} & \text{for }l_1-\epsilon<t\leq l_1, \\
    1, & \text{for } l_1<t\leq l_2, \\
    1-\frac{t-l_{2}}{\epsilon}, & \text{for }l_2<t\leq l_2+\epsilon, \\
    0, & \text{for } l_2 +\epsilon<t\leq t_2,
  \end{cases}
\]

and we choose 
\[
\phi(x,t)=w(x,t)\psi(x)^p\zeta_{\epsilon}(t)\eta(t)^p
\] 
as a test function in \eqref{weaksubsupsoln}. We denote 
\[
v_{h}^{p-1}=(u^{p-1})_h
\quad\text{and}\quad
\mathcal{V}(u(x,y,t))=\mathcal{A}(u(x,y,t))K(x,y,t).
\] 
Then following \cite{Verenacontinuity,KLin}, we observe that the subsolution $u$ of \eqref{maineqn} satisfies the following mollified inequality
\begin{equation}\label{mollified}
I_{h,\epsilon}+J_{h,\epsilon}\leq 0,
\end{equation}
where
$$
I_{h,\epsilon}=\int_{t_1}^{t_2}\int_{B_r}\partial_t{v_{h}^{p-1}}\phi(x,t)\,dx\,dt
=\int_{t_1}^{t_2}\int_{B_r}\partial_t{v_{h}^{p-1}}w(x,t)\psi(x)^p\zeta_{\epsilon}(t)\eta(t)^p\,dx\, dt,
$$
and
\begin{align*}
J_{h,\epsilon}&=\int_{t_1}^{t_2}\int_{\mathbb{R}^n}\int_{\mathbb{R}^n}\big(\mathcal{V}(u(x,y,t))\big)_{h}(\phi(x,t)-\phi(y,t))\,dx\,dy\,dt\\
&=\int_{t_1}^{t_2}\int_{\mathbb{R}^n}\int_{\mathbb{R}^n}\big(\mathcal{V}(u(x,y,t))\big)_{h}\big(w(x,t)\psi(x)^p-w(y,t)\psi(y)^p\big)\zeta_{\epsilon}(t)\eta(t)^p\,dx\,dy\,dt.
\end{align*}
\textbf{Estimate of $I_{h,\epsilon}$:} Proceeding similarly as in the proof of \cite[Proposition 3.1]{Verenacontinuity}, we have 
\begin{equation}\label{Ihpre}
\begin{split}
\lim_{\epsilon\to 0}\lim_{h\to 0}I_{h,\epsilon}\geq \int_{B_r}\xi(w)(x,l_2)\psi(x)^p\eta^p(l_2)\,dx-\int_{B_r}\xi(w)(x,l_1)\psi(x)^p\eta(l_1)^p\,dx\\
-\int_{l_1}^{l_2}\int_{B_r}\xi(w)(x,t)\psi(x)^p\partial_t\eta(t)^p\,dx\, dt.
\end{split}
\end{equation}

\textbf{Estimate of $J_{h,\epsilon}$:} First we claim that 
$
\lim_{h\to 0}J_{h,\epsilon}=J_{\epsilon},
$
where
$$
J_{\epsilon}=\int_{t_1}^{t_2}\int_{\mathbb{R}^n}\int_{\mathbb{R}^n}\mathcal{V}(u(x,y,t))\big(w(x,t)\psi(x)^p-w(y,t)\psi(y)^p\big)\zeta_{\epsilon}(t)\eta(t)^p\,dx\,dy\,dt.
$$
Indeed, we can write 
\begin{equation}\label{nonlocal}
J_{h,\epsilon}-J_{\epsilon}=L_{h,\epsilon}+N_{h,\epsilon},
\end{equation}
where 
$$
L_{h,\epsilon}=\int_{t_1}^{t_2}\int_{B_r}\int_{B_r}\big(\big(\mathcal{V}(u(x,y,t))\big)_h-\mathcal{V}(u(x,y,t))\big)\big(w(x,t)\psi(x)^p-w(y,t)\psi(y)^p\big)\zeta_{\epsilon}(t)\eta(t)^p\,dx\,dy\,dt,
$$
and
$$
N_{h,\epsilon}=2 \int_{t_1}^{t_2}\int_{B_r}\int_{\mathbb{R}^n\setminus B_r}\big(\big(\mathcal{V}(u(x,y,t))\big)_h-\mathcal{V}(u(x,y,t))\big)w(x,t)\psi(x)^p\zeta_{\epsilon}(t)\eta(t)^p\,dx\,dy\,dt.
$$
\textbf{Estimate of $L_{h,\epsilon}$:} We can rewrite $L_{h,\epsilon}$ as
\begin{align*}
L_{h,\epsilon}=\int_{t_1}^{t_2}\int_{B_r}\int_{B_r}\big(\big(\mathcal{V}(u(x,y,t))\big)_h-\mathcal{V}(u(x,y,t))\big)
\frac{\big(w(x,t)\psi(x)^p-w(y,t)\psi(y)^p\big)\zeta_{\epsilon}(t)\eta(t)^p}{|x-y|^{-\frac{(n+ps)}{p}}|x-y|^\frac{n+ps}{p}}\,dx\,dy\,dt,
\end{align*}
and using H\"older's inequality with exponents $p'=\frac{p}{p-1}$ and $p$, we obtain
\begin{equation}\label{loclimit}
\begin{split}
L_{h,\epsilon}\leq\Big(\int_{t_1}^{t_2}\int_{B_r}\int_{B_r}\Big|\big(\big(\mathcal{V}(u(x,y,t))\big)_h-\mathcal{V}(u(x,y,t))\big)|x-y|^\frac{n+ps}{p}\Big|^{p'}\,dx\,dy\,dt\Big)^\frac{1}{p'}\\
\cdot\Big(\int_{t_1}^{t_2}\int_{B_r}\int_{B_r}\frac{\big|\big(w(x,t)\psi(x)^p-w(y,t)\psi(y)^p\big)\zeta_{\epsilon}(t)\eta(t)^p\big|^p}{|x-y|^{n+ps}}\,dx\,dy\,dt\Big)^\frac{1}{p}.
\end{split}
\end{equation}
Now using the property \eqref{kernel}, we observe that, 
$$
|x-y|^\frac{n+ps}{p}|\mathcal{V}(u(x,y,t)|\leq\Lambda\frac{|u(x,t)-u(y,t)|^{p-1}}{|x-y|^\frac{n+ps}{p'}}\in L^{p'}((t_1,t_2)\times B_r\times B_r),
$$
From \cite[Lemma 2.2]{KLin}, we have
$$
\big(\big(\mathcal{V}(u(x,y,t))\big)_h-\mathcal{V}(u(x,y,t))\big)|x-y|^\frac{n+ps}{p}\to 0 \text{ in } L^{p'}((t_1,t_2)\times B_r\times B_r),
$$
and therefore from \eqref{loclimit}, it follows that  $\lim_{h\to 0}L_{h,\epsilon}=0$.\\
\textbf{Estimate of $N_{h,\epsilon}$:} We note that given the pointwise convergence of mollified functions together with the fact that $u\in L^{\infty}((t_1,t_2);L^{\infty}(\mathbb{R}^n))$, we can therefore apply the Lebesgue dominated convergence theorem to conclude that $\lim_{h\to 0}N_{h,\epsilon}=0$.\\
\textbf{Estimate of $J_{\epsilon}$:} We can rewrite $J_{\epsilon}=J^{1}_{\epsilon}+J^{2}_{\epsilon}$, where
\[
J^{1}_{\epsilon}=\int_{t_1}^{t_2}\int_{B_r}\int_{B_r}\mathcal{A}(u(x,y,t))(w(x,t)\psi(x)^p-w(y,t)\psi(y)^p)\zeta_{\epsilon}(t)\eta(t)^p\,d\mu\,dt,
\]
and
\[
J^{2}_{\epsilon}=2\int_{t_1}^{t_2}\int_{\mathbb{R}^n\setminus B_r}\int_{B_r}\mathcal{A}(u(x,y,t))w(x,t)\psi(x)^p\zeta_{\epsilon}(t)\eta(t)^p\,d\mu\,dt.
\]
\textbf{Estimate of $J^{1}_{\epsilon}$:} To estimate the integral $J^{1}_{\epsilon}$, we mainly adapt an idea from the  proof of \cite[Theorem 1.4]{Kuusilocal}. By symmetry we may assume $u(x,t)\geq u(y,t)$. In this case, for every fixed $t$, we observe that
\begin{align*}
&|u(x,t)-u(y,t)|^{p-2}(u(x,t)-u(y,t))\big(w(x,t)\psi(x)^p-w(y,t)\psi(y)^p\big)\eta(t)^p\\
&=(u(x,t)-u(y,t))^{p-1}\big(w(x,t)\psi(x)^p-w(y,t)\psi(y)^p\big)\eta(t)^p\\
&=\begin{cases}
(w(x,t)-w(y,t))^{p-1}\big(w(x,t)\psi(x)^p-w(y,t)\psi(y)^p\big)\eta(t)^p,\text{ if }u(x,t),u(y,t)>k,\\
(u(x,t)-u(y,t))^{p-1}w(x,t)\psi(x)^p\eta(t)^p,\text{ if }u(x,t)>k,\,u(y,t)\leq k,\\
0,\text{ otherwise. }
\end{cases}
\end{align*}
Thus 
\begin{align*}
&|u(x,t)-u(y,t)|^{p-2}(u(x,t)-u(y,t))\big(w(x,t)\psi(x)^p-w(y,t)\psi(y)^p\big)\eta(t)^p\\
&\qquad\geq |w(x,t)-w(y,t)|^{p-1}(w(x,t)\psi(x)^p-w(y,t)\psi(y)^p)\eta(t)^p.
\end{align*}
This implies,
\[
J^{1}_{\epsilon}\geq \int_{t_1}^{t_2}\int_{B_r}\int_{B_r}(w(x,t)-w(y,t))^{p-1}(w(x,t)\psi(x)^p-w(y,t)\psi(y)^p)\zeta_{\epsilon}(t)\eta(t)^p\,d\mu\, dt.
\]
Let us now consider the case when $w(x,t)>w(y,t)$ and $\psi(x)\leq\psi(y)$. By Lemma \ref{epsilon} we obtain
\begin{equation}\label{epsilonequationinequality}
\psi(x)^p\geq (1-C(p)\epsilon)\psi(y)^p-(1+C(p)\epsilon)\epsilon^{1-p}|\psi(x)-\psi(y)|^p
\end{equation}
for any $\epsilon\in(0,1]$ where $C(p)=(p-1)\Gamma(\max\{1,p-2\})$. Now by letting
$$
\epsilon=\frac{1}{\max\{1,2C(p)\}}\frac{w(x,t)-w(y,t)}{w(x,t)}\in(0,1],
$$
we deduce from above that the following inequality holds for some positive constant $C=C(p)$,
\begin{align*}
&(w(x,t)-w(y,t))^{p-1}w(x,t)\psi(x)^p
\geq (w(x,t)-w(y,t))^{p-1}w(x,t)\max\{\psi(x),\psi(y)\}^p\\
&\qquad-\frac{1}{2}(w(x,t)-w(y,t))^p\max\{\psi(x),\psi(y)\}^p-C\max\{w(x,t),w(y,t)\}^p|\psi(x)-\psi(y)|^p.
\end{align*}
Note that over here, we  used that under the assumption $\psi(x)\leq\psi(y)$, we have $\max\{\psi(x),\psi(y)\}=\psi(y)$. In the other cases $w(x,t)\geq w(y,t)$, $\psi(x)\geq\psi(y)$ or $w(x,t)=w(y,t)$, the above estimate is clear. Therefore, when $w(x,t)\geq w(y,t)$, we have
\begin{equation}\label{subballpartestimate1}
\begin{split}
&(w(x,t)-w(y,t))^{p-1}(w(x,t)\psi(x)^p-w(y,t)\psi(y)^p)\\
&\geq (w(x,t)-w(y,t))^{p-1}(w(x,t)\max\{\psi(x),\psi(y)\}^p-w(y,t)\psi(y)^p)\\
&-\frac{1}{2}(w(x,t)-w(y,t))^p\max\{\psi(x),\psi(y)\}^p
-C\max\{w(x,t),w(y,t)\}^p|\psi(x)-\psi(y)|^p\\
&\geq \frac{1}{2}(w(x,t)-w(y,t))^p\max\{\psi(x),\psi(y)\}^p
-C\max\{w(x,t),w(y,t)\}^p|\psi(x)-\psi(y)|^p.
\end{split}
\end{equation}
If $w(x,t)<w(y,t)$, we may interchange the roles of $x$ and $y$ above to obtain \eqref{subballpartestimate1}. We then observe that 
\begin{equation}\label{subballpartestimate2}
\begin{gathered}
|w(x,t)\psi(x)-w(y,t)\psi(y)|^p\leq 2^{p-1}|w(x,t)-w(y,t)|^{p}\max\{\psi(x),\psi(y)\}^p\\
+2^{p-1}\max\{w(x,t),w(y,t)\}^p|\psi(x)-\psi(y)|^p.
\end{gathered}
\end{equation}
Now \eqref{subballpartestimate1} and \eqref{subballpartestimate2} gives 
\begin{equation}\label{subballpartestimate3}
\begin{split}
J_{\epsilon}^{1}&\geq c\int_{t_1}^{t_2}\int_{B_r}\int_{B_r}|w(x,t)\psi(x)-w(y,t)\psi(y)|^p\zeta_{\epsilon}(t)\eta(t)^p\,d\mu\, dt\\
&\qquad-C\int_{t_1}^{t_2}\int_{B_r}\int_{B_r}\max\{w(x,t),w(y,t)\}^p|\psi(x)-\psi(y)|^p\zeta_{\epsilon}(t)\eta(t)^p\,d\mu\, dt,
\end{split}
\end{equation}
for some positive constants $c=c(p), C=C(p)$.\\
\textbf{Estimate of $J_{\epsilon}^{2}$:} To estimate $J_{\epsilon}^{2}$, we observe that
\begin{align*}
|u(x,t)-u(y,t)|^{p-2}(u(x,t)-u(y,t))w(x,t)
&\geq -(u(y,t)-u(x,t))^{p-1}w(x,t)\\
&\geq -(u(y,t)-k)^{p-1}_{+}w(x,t)\\
&\geq -w(y,t)^{p-1}w(x,t). 
\end{align*}
As a consequence, we obtain,
\begin{equation}\label{subballpartestimate4}
\begin{split}
J_{\epsilon}^{2}&\geq -\int_{t_1}^{t_2}\int_{\mathbb{R}^n\setminus B_r}\int_{B_r}K(x,y,t)w(y,t)^{p-1}w(x,t)\psi(x)^p\zeta_{\epsilon}(t)\eta(t)^p\,dx\,dy\,dt\\
&\geq -\Lambda\Big(\sup_{{t_1<t<t_2},\,x\in\mathrm{supp}\,\psi}\int_{\mathbb{R}^n\setminus B_r}\frac{w(y,t)^{p-1}}{|x-y|^{n+ps}}\,dy\Big)
\int_{t_1}^{t_2}\int_{B_r}w(x,t)\psi(x)^p\zeta_{\epsilon}(t)\eta(t)^p\,dx\,dt.
\end{split}
\end{equation}
Therefore from \eqref{subballpartestimate3} and \eqref{subballpartestimate4}, we obtain for some positive constants $c=c(p)$ and $C=C(p)$, 
\begin{equation}\label{jlimit}
\begin{split}
&\lim_{\epsilon\to 0}\lim_{h\to 0}J_{h,\epsilon}
=\lim_{\epsilon\to 0}J_{\epsilon}
=\lim_{\epsilon\to 0}(J_{\epsilon}^{1}+J_{\epsilon}^{2})\\
&\geq c\int_{l_1}^{l_2}\int_{B_r}\int_{B_r}|w(x,t)\psi(x)-w(y,t)\psi(y)|^p\eta(t)^p\,d\mu\, dt\\
&\qquad-C\int_{l_1}^{l_2}\int_{B_r}\int_{B_r}\max\{w(x,t),w(y,t)\}^p|\psi(x)-\psi(y)|^p\eta(t)^p\,d\mu\,dt\\
&\qquad-\Lambda\Big(\sup_{{t_1<t<t_2},\,x\in\mathrm{supp}\,\psi}\int_{\mathbb{R}^n\setminus B_r}\frac{w(y,t)^{p-1}}{|x-y|^{n+ps}}\,dy\Big)
\int_{l_1}^{l_2}\int_{B_r}w(x,t)\psi(x)^p\eta(t)^p\,dx\,dt.
\end{split}
\end{equation}
Now employing the estimates \eqref{Ihpre} and \eqref{jlimit} into \eqref{mollified} and then first letting $l_1\to t_1$ and then  by $l_2\to t_2$, we get
\begin{equation}\label{seminorm1}
\begin{split}
&\int_{t_1}^{t_2}\int_{B_r}\int_{B_r}|w(x,t)\psi(x)-w(y,t)\psi(y)|^p \eta(t)^{p}\,d\mu\,dt\\
&\leq C\Bigg(\int_{t_1}^{t_2}\int_{B_r}\int_{B_r}{\max\{w(x,t),w(y,t)\}^p|\psi(x)-\psi(y)|^p}\eta(t)^p\,d\mu\,dt\\
&\qquad+\Big(\sup_{x\in\mathrm{supp}\,\psi,\,t_1<t<t_2}\int_{{\mathbb{R}^n\setminus B_r}}{\frac{w(y,t)^{p-1}}{|x-y|^{n+ps}}}\,dy\Big)
\int_{t_1}^{t_2}\int_{B_r}w(x,t)\psi(x)^p\eta(t)^p\,dx\,dt\\
&\qquad+\int_{t_1}^{t_2}\int_{B_r}\xi(w)\psi(x)^p\partial_t\eta(t)^p\,dx\,dt\Bigg).
\end{split}
\end{equation}
Again using \eqref{Ihpre} and \eqref{jlimit} and then first letting $l_1\to t_1$ and then  by choosing $l_2 \in (\tau_1, \tau_2)$ such that
\[
\int_{B_r}\xi(w)(x,l_2)\psi(x)^p\,dx \geq \frac{1}{2} \sup_{\tau_1 < t< \tau_2} \int_{B_r}\xi(w)(x,t)\psi(x)^p\,dx,
\]
  we observe that
\begin{equation}\label{parabolicnorm1}
\begin{split}
&\sup_{\tau_1 < t< \tau_2} \int_{B_r}\xi(w)(x,t)\psi(x)^p\,dx\\
&\leq C\Bigg(\int_{t_1}^{t_2}\int_{B_r}\int_{B_r}{\max\{w(x,t),w(y,t)\}^p|\psi(x)-\psi(y)|^p}\eta(t)^p\,d\mu\,dt\\
&\qquad+\Big(\sup_{x\in\mathrm{supp}\,\psi,\,t_1<t<t_2}\int_{{\mathbb{R}^n\setminus B_r}}{\frac{w(y,t)^{p-1}}{|x-y|^{n+ps}}}\,dy\Big)
\int_{t_1}^{t_2}\int_{B_r}w(x,t)\psi(x)^p\eta(t)^p\,dx\,dt\\
&\qquad+\int_{t_1}^{t_2}\int_{B_r}\xi(w)\psi(x)^p\partial_t\eta(t)^p\,dx\,dt\Bigg).
\end{split}
\end{equation}
Now from \eqref{seminorm1} and \eqref{parabolicnorm1}, we get the required estimate.
\end{proof}

\section{Proof of Theorem \ref{Localbound1}}
Let us first assume that  $sp<n$ and for $j\in\mathbb{N}$, we denote by 
$$
r_0=r,
\quad r_j =\frac{1+2^{-j}}{2}r,
\quad s_j =\frac{r_j+r_{j+1}}{2},
$$
and
$$
B_j=B_{r_j}(x_0), 
\quad \bar{B}_j=B_{s_j}(x_0),
\quad \Gamma_j=(t_0-r_j^{ps},t_0),
\quad \bar{\Gamma}_j=(t_0-{s}_j^{ps},t_0).
$$
Moreover, for $\bar{k}>0$ to be chosen later, we let
$$
k_j=(1-2^{-j})\bar{k},\quad\bar{k}_j=\frac{k_{j+1}+k_j}{2},\quad w_j=(u-k_j)_{+}
\quad\text{and}\quad
\bar{w}_j=(u-\bar{k}_j)_{+}.
$$
We observe that
$$
\bar{k}_j>k_j,
\quad w_j\geq\bar{w_j}
\quad\text{and}\quad
 w_j^p\geq(2^{-j-2}\bar{k})^{p-1}\bar{w_j}=(\bar{k}_j-k_j)^{p-1}\bar{w}_j.
$$
Moreover, we choose $\psi_j\in C_{0}^\infty(B_j),\,\eta_j\in C^{\infty}(\Gamma_j)$
such that $0\leq\psi_j\leq 1$ in $B_{j}$, $\psi_j\equiv 1$ on $B_{j+1}$, $|\nabla\psi_j|<\frac{2^{j+3}}{r}$ in $B_j$ and $0\leq\eta_j\leq 1$ in $\Gamma_{j}$, and $\eta_j\equiv 1$ on $\Gamma_{j+1}$ with $\eta_{j}=0$ on $\Gamma_{j}\setminus\bar{\Gamma}_{j}$ and  $|\partial_t\eta_j|\leq\frac{2^{jps}}{r^{ps}}\text{ in }\Gamma_j$. 
By Lemma \ref{parabolic embedding} with $\kappa=\frac{n+ps}{n}$ and $\kappa^{*}=\frac{n}{n-sp}$, we have
\begin{equation}\label{Embapp}
\begin{split}
&\int_{\Gamma_{j+1}}\fint_{B_{j+1}}|\bar{w_j}|^{p\kappa}\,dx\,dt\\
&\leq Cr_{j+1}^{sp-n}\int_{\Gamma_{j+1}}\int_{B_{j+1}}\int_{B_{j+1}}\frac{|\bar{w}_j(x,t)-\bar{w}_j(y,t)|^p}{|x-y|^{n+sp}}\,dx\,dy\,dt\cdot\Big(\sup_{\Gamma_{j+1}}\fint_{B_{j+1}}|\bar{w}_{j}|^p\,dx\Big)^\frac{ps}{n}\\
&\qquad+C\int_{\Gamma_{j+1}}\fint_{B_{j+1}}\bar{w}_j^p\,dx\,dt\cdot\Big(\sup_{\Gamma_{j+1}}\fint_{B_{j+1}}|\bar{w}_j|^p\,dx\Big)^\frac{ps}{n}\\
&=Cr_{j+1}^{sp-n}I_1\cdot\Big(\frac{I_2}{|B_{j+1}|}\Big)^\frac{ps}{n}+C\int_{\Gamma_{j+1}}\fint_{B_{j+1}}|\bar{w}_j|^p\,dx\,dt\cdot\Big(\frac{I_2}{|B_{j+1}|}\Big)^\frac{ps}{n},
\end{split}
\end{equation}
where
\begin{align*}
I_1=\int_{\Gamma_{j+1}}\int_{B_{j+1}}\int_{B_{j+1}}\frac{|\bar{w}_j(x,t)-\bar{w}_j(y,t)|^p}{|x-y|^{n+sp}}\,dx\,dy\,dt
\quad\text{and}\quad 
I_2=\sup_{\Gamma_{j+1}}\int_{B_{j+1}}|\bar{w}_j|^p\,dx.
\end{align*}
Let $U_j=B_j\times\Gamma_j$ and $\bar{U}_j=\bar{B}_{j}\times\bar{\Gamma}_j$.
Since $r_{j+1}<r_j,\,s_j<r_j$, we have $\bar{B_j}\subset B_j$, $\bar{\Gamma}_j\subset\Gamma_{j}$,
$B_{j+1}\subset B_j$ and $\Gamma_{j+1}\subset\Gamma_j$.
To estimate $I_1$ and $I_2$ we apply Lemma \ref{Subenergyestimate} with $r={r_j}$, $\tau_2=t_0$, $\tau_1=t_0-r_{j+1}^{ps}$, $l=s_j^{ps}-r_{j+1}^{ps}$ and $\phi_j(x,t)=\psi_j(x)\eta_{j}(t)$ with 
$\eta_j(t)=0\text{ if }t\leq\tau_1-l$ and $\eta_j(t)=1$\text{ if }$t\geq \tau_1$. Observing that
$B_{j+1}\subset\bar{B_j}$ and $\Gamma_{j+1}\subset\bar{\Gamma}_{j}$,
using Lemma \ref{Subenergyestimate}, for some positive constant $C=C(n,p,s,\Lambda)$ we get
\begin{equation}\label{I12}
\begin{split}
I_1+C\,I_2
&\leq\int_{\Gamma_j}\int_{B_j}\int_{B_j}|\bar{w}_j(x,t)\psi_j(x)-\bar{w}_j(y,t)\psi_j(y)|^p\eta_j(t)^p\,d\mu\,dt\\
&\qquad+C\sup_{\Gamma_{j+1}}\int_{B_j}\bar{w}_j(x,t)^{p}\psi_{j}(x)^p\,dx\\
&\leq C(J_1+J_2+J_3),
\end{split}
\end{equation}
where
$$
J_1=\int_{\Gamma_j}\int_{B_j}\int_{B_j}\max\{\bar{w}_j(x,t)^p,\bar{w}_j(y,t)^p\}|\psi_j(x)-\psi_{j}(y)|^p\eta_j(t)^p\,d\mu\,dt,
$$
$$
J_2=\sup_{t\in\Gamma_j,\,x\in\text{supp}\,\psi_j}\int_{\mathbb{R}^n\setminus B_j}
\frac{\bar{w}_j(y,t)^{p-1}}{|x-y|^{n+ps}}\,dy\int_{B_j}\bar{w}_j(x,t)\psi_j(x)^p\eta_j(t)^p\,dx,
$$
and
$$
J_3=\int_{\Gamma_j}\int_{B_j}\xi(\bar{w}_j)(x,t)\psi_j(x)^p\partial_t\eta_j(t)^p\,dx dt.
$$
Now we estimate each $J_i$, $i=1,2,3$ separately. \\
\textbf{Estimate of $J_1$:} Using $r_j<r$ and $\bar{w}_j\leq w_j$, we have
\begin{equation}\label{JI}
\begin{split}
J_1&=\int_{\Gamma_j}\int_{B_j}\int_{B_j}\max\{\bar{w}_j(x,t)^p,\bar{w}_j(y,t)^p\}|\psi_j(x)-\psi_{j}(y)|^p\eta_j(t)^p\,d\mu\,dt\\
&\leq C(n,p,s,\Lambda)\Big(\sup_{x\in B_j}\int_{B_j}\frac{|\psi_j(x)-\psi_j(y)|^p}{|x-y|^{n+sp}}\,dy\Big)\int_{\Gamma_j}\int_{B_j}\bar{w}_j(x,t)^p\,dx\,dt\\
&\leq C(n,p,s,\Lambda)\frac{2^{j(n+sp+p)}}{r_j^{sp}}\int_{\Gamma_j}\int_{B_j}w_{j}(x,t)^p\,dx\,dt.
\end{split}
\end{equation}
\textbf{Estimate of $J_2$:} Without loss of generality, we may assume $x_0=0$. Using the fact that $\bar{w}_j\leq w_0$, under the assumptions on $\psi_j$, we have for $x\in\text{supp}\,\psi_j$, and $y\in\mathbb{R}^n\setminus B_j$,
$$
\frac{1}{|x-y|}=\frac{1}{|y|}\frac{|x-(x-y)|}{|x-y|}\leq\frac{1}{|y|}\big(1+2^{j+3}\big)\leq\frac{2^{j+4}}{|y|}.
$$
This implies
\begin{equation}\label{subenergyJ2}
\begin{split}
J_2&=\sup_{t\in\Gamma_j,\,x\in\text{supp}\,\psi_j}\int_{\mathbb{R}^n\setminus B_j}\frac{\bar{w}_j(y,t)^{p-1}}{|x-y|^{n+ps}}\,dy
\int_{\Gamma_j}\int_{B_j}\bar{w}_j(x,t)\psi_j(x)^p\eta_j(t)^p\,dx\,dt\\
&\leq C\frac{2^{j(n+sp+p)}}{r^{sp}\bar{k}^{p-1}}\Tail_{\infty}^{p-1}(w_0;x_0,\frac{r}{2},t_0-r^{sp},t_0)\int_{\Gamma_j}\int_{B_j}w_j(x,t)^{p}\,dx\,dt\\
&\leq C\frac{2^{j(n+sp+p)}}{\delta^{p-1}r_{j}^{sp}}\int_{\Gamma_j}\int_{B_j}w_j(x,t)^{p}\,dx\,dt,
\end{split}
\end{equation}
where we also used the fact $\bar{w}_j\leq\Big(\frac{2^{j+2}}{\bar{k}}\Big)^{p-1}w_j^{p}$ and  also that  $\bar{k}$ would be chosen finally such that  $\bar{k}\geq\delta\,\Tail_{\infty}(w_0;x_0,\frac{r}{2},t_0-r^{sp},t_0)$.\\
\textbf{Estimate of $J_3$:} To estimate $J_3$, we first note that by  Lemma \ref{Auxfnlemma} and the fact that $p>2$ we have,
\begin{equation}\label{J3}
\begin{split}
J_3&=\int_{\Gamma_j}\int_{B_j}\xi(\bar{w}_j)(x,t)\psi_{j}(x)^p\partial_t\eta_{j}(t)^p\,dx\,dt\\
&\leq C(p)\int_{\Gamma_j}\int_{B_j}(\bar{w}_j(x,t)+\bar{k}_j)^{p-2}\bar{w}_j(x,t)^{2}\psi_j(x)^p|\partial_t\eta_j(t)^{p}|\,dx\,dt\\
&=J_4+J_5,
\end{split}
\end{equation}
where
$$
J_4=\int_{(\Gamma_j\times B_j)\cap\{0<u-\bar{k}_j<\bar{k}_j\}}(\bar{w}_j(x,t)+\bar{k}_j)^{p-2}\bar{w}_j(x,t)^{2}\psi_j(x)^p|\partial_t\eta_j(t)^{p}|\,dx\,dt,
$$
and
\begin{equation}\label{J5}
\begin{split}
J_5&=\int_{(\Gamma_j\times B_j)\cap{\{\bar{w}_j\geq \bar{k}_j}\}}(\bar{w}_j(x,t)+\bar{k}_j)^{p-2}\bar{w}_j(x,t)^{2}\psi_j(x)^p|\partial_t\eta_j(t)^{p}|\,dx\,dt\\
&\leq2^{p-2}\int_{(\Gamma_j\times B_j)\cap{\{\bar{w}_j\geq \bar{k}_j}\}}\bar{w}_j(x,t)^p\psi_j(x)^p|\partial_t\eta_j(t)^p|\,dx\,dt\\
&\leq C(p,s)\frac{2^{j(n+sp+p)}}{r_j^{ps}}\int_{\Gamma_j}\int_{B_j}w_j(x,t)^p\,dx\,dt,
\end{split}
\end{equation}
where to deduce the estimate \eqref{J5} we have again used the fact that $p>2$.\\
\textbf{Estimate of $J_4$:} Now we estimate $J_4$  by adapting  some ideas from \cite{Gvespri}. Indeed, we denote by $A_j=(\Gamma_j\times B_j)\cap\{0<u-\bar{k}_j<\bar{k}_j\}$ and using binomial expansion  we observe that,
\begin{equation}\label{J4}
\begin{split}
J_4&=\int_{A_j}(\bar{w}_j(x,t)+\bar{k}_j)^{p-2}\bar{w}_j(x,t)^{2}\psi_j(x)^p|\partial_t\eta_j(t)^{p}|\,dx\,dt\\
&=\sum_{d=0}^{\infty}\int_{A_j}\binom{p-2}{d}\bar{k}_j^{p}\left(\frac{\bar{w}_j(x,t)}{\bar{k}_j}\right)^{d+2}|\partial_t\eta_j(t)^{p}|\,dx\,dt\\
&=J_4^{1}+J_4^{2},
\end{split}
\end{equation}
where
$$
J_4^{1}=\sum_{d=0}^{[p-2]}\int_{A_j}\binom{p-2}{d}\bar{k}_j^{p}\left(\frac{\bar{w}_j(x,t)}{\bar{k}_j}\right)^{d+2}|\partial_t\eta_j(t)^{p}|\,dx\,dt,
$$
and
$$
J_4^{2}=\sum_{d=[p-2]+1}^{\infty}\int_{A_j}\binom{p-2}{d}\bar{k}_j^{p}\left(\frac{\bar{w}_j(x,t)}{\bar{k}_j}\right)^{d+2}|\partial_t\eta_j(t)^{p}|\,dx\,dt.
$$
\textbf{Estimate of $J_4^{1}$:} Let us estimate $J_4^{1}$ as follows. Using H$\ddot{\text{o}}$lder's inequality, we obtain 
$$
J_4^{1}\leq\sum_{d=0}^{[p-2]}\left|\binom{p-2}{d}\right|(\bar{k}_j)^{p-2-d}\Big(\int_{A_j}\bar{w}_j^{p}(\partial_t\eta^p_{j})^\frac{p}{d+2}\,dx\,dt\Big)^\frac{d+2}{p}|A_j|^{1-\frac{d+2}{p}}.
$$
Now, since $u\geq\bar{k}_j$ in $A_j$, we observe that,
$$
\int_{A_j}w_j^p\,dx\,dt\geq\Big(\bar{k}_j-k_j\Big)^p|A_j|=\Big(\frac{\bar{k}}{2^{j+2}}\Big)^p|A_j|.
$$
Therefore, we obtain
\begin{equation}\label{fact}
|A_j|\leq\Big(\frac{2^{j+2}}{\bar{k}}\Big)^p\int_{A_j}w_j^p\,dx\,dt.
\end{equation}
Now using  \eqref{fact} together with the fact  $\bar{w}_j\leq w_j$, $\bar{k}_j<\bar{k}$, $r_j<r$ and also by using  the bounds on  $|\partial_t \eta_j|$, we get
\begin{equation}\label{J41}
\begin{split}
J_4^{1}&\leq C(p)\sum_{d=0}^{[p-2]}\left|\binom{p-2}{d}\right|2^{jp}\Big(\int_{A_j}w_j^{p}|\partial_t \eta_j|^\frac{p}{d+2}\,dx\,dt\Big)^\frac{d+2}{p}
\Big(\int_{A_j}w_j^{p}\,dx\,dt\Big)^{1-\frac{d+2}{p}}\\
&\leq C(p)\frac{2^{jp(s+1)}}{r_j^{sp}}\int_{A_j}w_j^{p}\,dx dt.
\end{split}
\end{equation}
\textbf{Estimate of $J_4^{2}$:} Now since $\bar{w}_j<\bar{k}_j$, therefore for all $d\geq[p-2]+1$, we have that $\bar{w}_j^{d-[p-2]-1}\leq \bar{k}_j^{d-[p-2]-1}$. Thus  $\bar{k}_j^{p-2-d}\bar{w}_j^{d+2}\leq \bar{k}_{j}^{p-3-[p-2]}\bar{w}_j^{[p-2]+3}$ and consequently we obtain
$$
J_4^{2}\leq\sum_{d=[p-2]+1}^{\infty}\left|\binom{p-2}{d}\right|\int_{A_j}\bar{k}_j^{p-3-[p-2]}\bar{w}_j^{[p-2]+3}|\partial_t\eta_j^p|\,dx\,dt.
$$
Finally by using $\bar{k}_j^{p-3-[p-2]}<\bar{w}_j^{p-3-[p-2]}$, we have
\begin{equation}\label{J42}
\begin{split}
J_4^{2}&\leq\sum_{d=[p-2]+1}^{\infty}\left|\binom{p-2}{d}\right|\int_{A_j}\bar{w}_j^{p}|\partial_t\eta_j^p|\,dx\,dt\\
&\leq C(p)\frac{2^{jps}}{r_j^{ps}}\int_{A_j}\bar{w}_j^p\,dx\,dt\\
&\leq C(p)\frac{2^{jps}}{r_j^{ps}}\int_{\Gamma_j}\int_{B_j}{w}_j^p\,dx\,dt,
\end{split}
\end{equation}
where we have also used the fact that the series $\sum_{d=0}^{\infty}|\binom{p-2}{d}|$ is convergent. Therefore, using \eqref{J41} and \eqref{J42} into \eqref{J4}, we obtain
\begin{equation}\label{J4final}
J_4\leq C(p)\frac{2^{jp(s+1)}}{r_j^{ps}}\int_{\Gamma_j}\int_{B_j}{w}_j^p\,dx\,dt. 
\end{equation}

Now using the estimates \eqref{J5} and \eqref{J4final} in \eqref{J3} we conclude
\begin{equation}\label{J3final}
J_3\leq C(p,s)\frac{2^{j(n+ps+p)}}{r_j^{ps}}\int_{\Gamma_j}\int_{B_j}{w}_j^p\,dx\,dt. 
\end{equation}
Then using $\bar{w}_j^{p\kappa}\geq (2^{-j-2}\bar{k})^{p(\kappa-1)}w_{j+1}^p$ in \eqref{Embapp}, we get
\begin{equation}\label{subenergyI}
\begin{split}
I&=(2^{-j-2}\bar{k})^{p(\kappa-1)}\fint_{\Gamma_{j+1}}\fint_{B_{j+1}}w_{j+1}^p\,dx\,dt\\
&\leq\frac{Cr_{j+1}^{ps-n}}{|\Gamma_{j+1}|}I_1\cdot\Big(\frac{I_2}{|B_{j+1}|}\Big)^\frac{ps}{n}
+C\fint_{\Gamma_{j+1}}\fint_{B_{j+1}}w_{j}^p\,dx\,dt\cdot\Big(\frac{I_2}{|B_{j+1}|}\Big)^\frac{ps}{n}.
\end{split}
\end{equation}
Plugging the estimates \eqref{JI}, \eqref{subenergyJ2} and  \eqref{J3final} into \eqref{I12}, we have
\begin{equation}\label{subenergyI12final}
I_1,I_2\leq C(n,p,s,\Lambda)\frac{2^{j(n+sp+p)}}{\delta^{p-1} r_j^{sp}}\int_{\Gamma_{j}}\int_{B_j}w_j^{p}\,dx\,dt.
\end{equation}
Then using \eqref{subenergyI12final} in \eqref{subenergyI}, we get
\begin{equation*}\label{Ifinal}
I\leq C(n,p,s,\Lambda)\Big(\frac{2^{j(n+ps+p)}}{\delta^{p-1}}\fint_{\Gamma_{j}}\fint_{B_{j}}w_j^{p}\,dx\,dt\Big)^\frac{n+ps}{n}.
\end{equation*}
We now let
$$
A_j=\Big(\fint_{U_j}w_j^p\,dx\,dt\Big)^\frac{1}{p}.
$$
Then we have
$$
(2^{-j-2}\bar{k})^{p(\kappa-1)}A_{j+1}^p\leq C(n,p,s,\Lambda)\Big(\frac{2^{j(n+ps+p)}}{\delta^{p-1}}A_j^{p}\Big)^\frac{n+ps}{n}.
$$
Then for some positive constant $C=C(n,p,s,\Lambda)$ we have
\begin{align*}
\frac{A_{j+1}}{\bar{k}}&\leq\frac{C}{\bar{k}^\kappa}2^{j(\kappa-1)}\Big(\frac{2^{j(n+ps+p)}}{\delta^{p-1}}A_{j}^p\Big)^\frac{n+ps}{np}=C\frac{2^{j\big(\kappa-1+(n+ps+p)(\frac{\kappa}{p})\big)}}{\delta^{(p-1)\frac{\kappa}{p}}}\Big(\frac{A_j}{\bar{k}}\Big)^{\kappa}
\end{align*}
with $\kappa=\frac{n+ps}{n}$.
Noting that $w_0=u$, we now let 
$$
\bar{k}=\delta \Tail_{\infty}(u;x_0,\frac{r}{2},t_0-r^{ps},t_0)+C^\frac{n}{ps}2^{\frac{n^2}{p^2 s^2}{\big(\kappa-1+(n+ps+p)(\frac{\kappa}{p})\big)}}\delta^{-\frac{(p-1)\kappa n}{sp^{2}}}\Big(\fint_{U^{-}(r)}u^p\,dx\,dt\Big)^\frac{1}{p},
$$
such that  for
\[
\beta=\frac{ps}{n}, 
\quad c_0=\frac{C}{\delta^{(p-1)\frac{\kappa}{p}}}>1, 
\quad b=2^{\kappa-1+(n+ps+p)(\frac{\kappa}{p})}>1
\quad\text{and}\quad
Y_j=\frac{A_j}{\bar{k}},
\]
the hypothesis of Lemma \ref{Deg2} is satisfied and consequently we have that 

$$\sup_{U^{-}(\frac{r}{2})}\,u\leq\bar{k},$$
which proves  Theorem \ref{Localbound1}. In the case when  $sp\geq n$, the proof follows by similar arguments.

\section{Some qualitative and quantitative properties of supersolutions}
In this section, we prove some qualitative  and quantitative properties of supersolutions which are strictly bounded away from zero.  Throughout this section,  by a global supersolution  $u$ in $\mathbb{R}^n \times (0, T)$, we refer to a  bounded positive function $u$ which satisfies  the hypothesis of Definition \ref{subsupsolution} in $\Omega \times (0, T)$ where $\Omega$ is any bounded domain in $\mathbb{R}^n$.   

\medskip

The following lemma  is the nonlocal analogue of Lemma 3.1 in \cite{Kin-Kuusi} which states that the inverse of a supersolution is a subsolution. 
\begin{Lemma}
Let $p>2$ and $u\geq\rho>0$ in $\mathbb{R}^n\times(0,T)$ be a supersolution of \eqref{maineqn}, then $u^{-1}$ is a subsolution of \eqref{maineqn}.
\end{Lemma}
\begin{proof}
Let $v=u^{-1}$ and $\psi\in W^{1,p}_{\loc}(0,T;L^p(\Omega'))\cap L^p_{\loc}(0,T;W_{0}^{s,p}(\Omega'))$ be nonnegative. 
Since $u$ is a weak supersolution of \eqref{maineqn}, by formally choosing $\phi(x,t)=u(x,t)^{2(1-p)}\psi(x,t)$ as a test function in \eqref{weaksubsupsoln} which can be justified by mollifying in time  as in the proof of  Lemma \ref{Subenergyestimate}, we obtain for every $[t_1,t_2]\Subset(0,T)$,
\begin{equation}\label{Inveqn}
0\leq I_1+I_2,
\end{equation}
where
\begin{align*}
I_1&=\int_{\Omega'}u(x,t_2)^{p-1}\phi(x,t_2)\,dx-\int_{\Omega'}u(x,t_1)^{p-1}\phi(x,t_1)\,dx-\int_{t_1}^{t_2}\int_{\Omega'}u(x,t)^{p-1}\partial_t\phi(x,t)\,dx\,dt,\\
&=\int_{\Omega'}u(x,t_2)^{1-p}\psi(x,t_2)\,dx-\int_{\Omega'}u(x,t_1)^{1-p}\psi(x,t_1)\,dx-I_3,
\end{align*}
with
\begin{align*}
I_3&=\int_{t_1}^{t_2}\int_{\Omega'}u(x,t)^{p-1}\big(u(x,t)^{2(1-p)}\partial_t\psi(x,t)-2(p-1)\psi(x,t)u(x,t)^{1-2p}\partial_t u(x,t)\big)\,dx\,dt\\
&=\int_{t_1}^{t_2}\int_{\Omega'}u(x,t)^{1-p}\partial_t\psi(x,t)\,dx\, dt
-2(p-1)\int_{t_1}^{t_2}\int_{\Omega'}\psi(x,t)u(x,t)^{-p}\partial_t u(x,t)\,dx\, dt\\
&=\int_{t_1}^{t_2}\int_{\Omega'}u(x,t)^{1-p}\partial_t\psi(x,t)\,dx\, dt
-2\int_{t_1}^{t_2}\int_{\Omega'}u(x,t)^{1-p}\partial_t\psi(x,t)\,dx\, dt+2I_4,
\end{align*}
and
\begin{align*}
I_4&=\int_{\Omega'}u(x,t_2)^{1-p}\psi(x,t_2)\,dx-\int_{\Omega'}u(x,t_1)^{1-p}\psi(x,t_1)\,dx.
\end{align*}
We thus obtain from above, 
\begin{align*}
I_1&=-\Big(\int_{\Omega'}v(x,t_2)^{p-1}\psi(x,t_2)\,dx-\int_{\Omega'}v(x,t_1)^{p-1}\psi(x,t_1)\,dx-\int_{t_1}^{t_2}\int_{\Omega'}v^{p-1}\partial_t\psi\,dx\,dt\Big).
\end{align*}
Here
\begin{align*}
I_2&=\int_{t_1}^{t_2}\int_{\mathbb{R}^n}\int_{\mathbb{R}^n}\mathcal{A}(u(x,y,t)){(\phi(x,t)-\phi(y,t))}\,d\mu\, dt\\
&=\int_{t_1}^{t_2}\int_{\mathbb{R}^n}\int_{\mathbb{R}^n}|u(x,t)-u(y,t)|^{p-2}(u(x,t)-u(y,t))\\
&\qquad\cdot(u(x,t)^{2(1-p)}\psi(x,t)-u(y,t)^{2(1-p)}\psi(y,t))\,d\mu\, dt\\
&=-\int_{t_1}^{t_2}\int_{\mathbb{R}^n}\int_{\mathbb{R}^n}\mathcal{A}(v(x,y,t)){\Big(\Big(\frac{v(x,t)}{v(y,t)}\Big)^{p-1}\psi(x,t)-\Big(\frac{v(y,t)}{v(x,t)}\Big)^{p-1}\psi(y,t)\Big)}\,d\mu\,dt.
\end{align*}

Now we estimate $I_2$. Let us first  consider the case when $v(x,t)\geq v(y,t)$. In this case, we have
\begin{align*}
\mathcal{A}(v(x,y,t)){\Big(\Big(\frac{v(x,t)}{v(y,t)}\Big)^{p-1}\psi(x,t)-\Big(\frac{v(y,t)}{v(x,t)}\Big)^{p-1}\psi(y,t)\Big)}&\geq \mathcal{A}(v(x,y,t))\big(\psi(x,t)-\psi(y,t)\big).
\end{align*}
Likewise when $v(x,t)<v(y,t)$, we have
\begin{align*}
&\mathcal{A}(v(x,y,t)){\Big(\Big(\frac{v(x,t)}{v(y,t)}\Big)^{p-1}\psi(x,t)-\Big(\frac{v(y,t)}{v(x,t)}\Big)^{p-1}\psi(y,t)\Big)}\\
&=|v(y,t)-v(x,t)|^{p-2}(v(y,t)-v(x,t))\Big(\Big(\frac{v(y,t)}{v(x,t)}\Big)^{p-1}\psi(y,t)-\Big(\frac{v(x,t)}{v(y,t)}\Big)^{p-1}\psi(x,t)\Big)\\
&\geq\mathcal{A}(v(y,x,t))(\psi(y,t)-\psi(x,t)).
\end{align*}
Therefore in either case we obtain
\begin{align*}
I_2&\leq -\int_{t_1}^{t_2}\int_{\mathbb{R}^n}\int_{\mathbb{R}^n}\mathcal{A}(v(x,y,t))(\psi(x,t)-\psi(y,t))\,d\mu\, dt.  
\end{align*}
By inserting the above estimates for $I_1$ and $I_2$ into \eqref{Inveqn}, we get
\begin{align*}
\int_{\Omega'}v(x,t_2)^{p-1}\psi(x,t_2)\,dx-\int_{\Omega'}v(x,t_1)^{p-1}\psi(x,t_1)\,dx-\int_{t_1}^{t_2}\int_{\Omega'}v(x,t)^{p-1}\partial_t\psi(x,t)\,dx\,dt\\
+\int_{t_1}^{t_2}\int_{\mathbb{R}^n}\int_{\mathbb{R}^n}\mathcal{A}(v(x,y,t))(\psi(x,t)-\psi(y,t))\,d\mu\, dt\leq 0.
\end{align*}
Hence $v=u^{-1}$ is a subsolution of \eqref{maineqn}.
\end{proof}

Now we prove an energy estimate for strictly positive supersolutions of \eqref{maineqn} which is the key ingredient needed to  deduce reverse H\"older inequality for strictly positive supersolutions.
\begin{Lemma}\label{supenergy1}
Let $p>2$, $x_0\in\mathbb{R}^n$, $r>0$ and $\alpha\in(0,{p-1})$. Suppose that $u$ is a weak supersolution of \eqref{maineqn} with $u\geq\rho>0$ in $\mathbb{R}^n\times(\tau_1,\tau_2+l)$, $(\tau_1,\tau_2+l)\subset(0,T)$. Then there exists positive constants $C=C(n,p,s,\Lambda)$ and $c=c(p)$ large enough such that
\begin{equation*}
\begin{split}
&\frac{p-1}{\alpha}\sup_{\tau_1<t<\tau_2}\int_{B_r}\psi(x)^p u(x,t)^\alpha\,dx
+\frac{\zeta(\epsilon)}{c(p)}\int_{\tau_1}^{\tau_2+l}\int_{B_r}\int_{B_r}\big|\psi(x)u(x,t)^\frac{\alpha}{p}-\psi(y)u(y,t)^\frac{\alpha}{p}\big|^p \eta(t)\,d\mu\,dt\\
&\leq\Big(\zeta(\epsilon)+1+\frac{1}{\epsilon^{p-1}}\Big)\int_{\tau_1}^{\tau_2+l}\int_{B_r}\int_{B_r}|\psi(x)-\psi(y)|^p(u(x,t)^\alpha+u(y,t)^\alpha)\eta(t)\,d\mu\,dt\\
&\qquad+C(\Lambda)\sup_{x\in\text{ supp }\psi}\int_{\mathbb{R}^n\setminus B_r}\frac{dy}{|x-y|^{n+ps}}\int_{\tau_1}^{\tau_2+l}\int_{B_r}u(x,t)^\alpha\psi(x)^p\eta(t)\,dx\,dt\\
&\qquad+\frac{p-1}{\alpha}\int_{\tau_1}^{\tau_2+l}\int_{B_r}u(x,t)^\alpha\psi(x)^p|\partial_t\eta(t)|\,dx\,dt,
\end{split}
\end{equation*}
where $\epsilon=p-\alpha-1$ and $\zeta(\epsilon)=\frac{\epsilon p^p}{\alpha^p}$, if $\alpha\geq 1$ and $\zeta(\epsilon)=\frac{\epsilon p^p}{\alpha}$ if $\alpha\in(0,1)$. Moreover $\psi\in C_{0}^{\infty}(B_r)$ is taken to be  nonnegative and  $\eta\in C^\infty(\mathbb{R})$  is  also non-negative such that $\eta(t)= 1$ if $\tau_1\leq t\leq\tau_2$ and $\eta(t)= 0$ if $t\geq\tau_2+l$.
\end{Lemma}

\begin{proof}
Let $t_1\in(\tau_1,\tau_2)$ and $t_2=\tau_2+l$. We consider $\eta\in C^\infty(t_1,t_2)$ such that $\eta(t_2)=0$ and $\eta(t)= 1$ for all $t\leq t_1$. Let $\epsilon\in(0,p-1)$ and $\alpha=p-\epsilon-1$.
Then since $u$ is a strictly positive weak supersolution of \eqref{maineqn}, choosing $\phi(x,t)=u(x,t)^{-\epsilon}\psi(x)^p\eta(t)$ as a test function in \eqref{weaksubsupsoln} (which is again justified  by mollifying in time),  we obtain
\begin{equation}\label{supestimate1}
\begin{gathered}
0\leq I_1+I_2+2 I_3,
\end{gathered}
\end{equation}
where
$$
I_1=\int_{t_1}^{t_2}\int_{B_r}\frac{\partial}{\partial t}(u^{p-1})\phi(x,t)\,dx\,dt,
$$
$$
I_2=\int_{t_1}^{t_2}\int_{B_r}\int_{B_r}\mathcal{A}(u(x,y,t))(u(x,t)^{-\epsilon}\psi(x)^p-u(y,t)^{-\epsilon}\psi(y)^p)\eta(t)\,d\mu\,dt,
$$
and
$$
I_3=\int_{t_1}^{t_2}\int_{\mathbb{R}^n\setminus B_r}\int_{B_r}\mathcal{A}(u(x,y,t))u(x,t)^{-\epsilon}\psi(x,t)^p\eta(t)\,d\mu\,dt.
$$
We observe that for any $x\in B_r$ and $y\in\mathbb{R}^n\setminus B_r$,  we have that the integrand in $I_3$ is non-negative precisely in the set where $u(x, t) \geq u(y, t)$.  In view of this, we observe that $I_3$ can be estimated from above in the following way,
\begin{equation}\label{supI3final}
\begin{split}
I_3&=\int_{t_1}^{t_2}\int_{\mathbb{R}^n\setminus B_r}\int_{B_r}|u(x,t)-u(y,t)|^{p-2}(u(x,t)-u(y,t))u(x,t)^{-\epsilon}\psi(x)^p\eta(t)\,d\mu\,dt\\
&\leq\int_{t_1}^{t_2}\int_{\mathbb{R}^n\setminus B_r}\int_{B_r} u(x,t)^{p-\epsilon-1}\psi(x)^p\eta(t)\,d\mu\,dt\\
&\leq C(\Lambda)\sup_{x\in\text{supp}\,\psi}\int_{\mathbb{R}^n\setminus B_r}\frac{dy}{|x-y|^{n+ps}}\int_{t_1}^{t_2}\int_{B_r}u(x,t)^{p-\epsilon-1}\psi(x)^p\eta(t)\,dx\,dt.
\end{split}
\end{equation}
Then we note that  $I_2$ can be estimated using Lemma \ref{Inequality1} as follows,
\begin{equation}\label{I2sup}
\begin{split}
I_2&\leq -\frac{\zeta(\epsilon)}{C(p)}\int_{t_1}^{t_2}\int_{B_r}\int_{B_r}\big|\psi(x)u(x,t)^\frac{\alpha}{p}-\psi(y)u(y,t)^\frac{\alpha}{p}\big|^p\eta(t)\,d\mu\, dt\\
&\qquad+\Big(\zeta(\epsilon)+1+\frac{1}{\epsilon^{p-1}}\Big)\int_{t_1}^{t_2}\int_{B_r}\int_{B_r}|\psi(x)-\psi(y)|^p(u(x,t)^\alpha+u(y,t)^\alpha)\eta(t)\,d\mu\, dt.
\end{split}
\end{equation}
For $I_1$ we have
\begin{equation}\label{I1final}
\begin{split}
I_1&=-\frac{p-1}{p-\epsilon-1}\int_{B_r}u^{p-\epsilon-1}(x,t_1)\psi(x)^p\,dx\\
&\qquad-\frac{p-1}{p-\epsilon-1}\int_{t_1}^{t_2}\int_{B_r}u(x,t)^{p-\epsilon-1}\psi(x)^p\partial_{t}\eta(t)\,dx\,dt.
\end{split}
\end{equation}

Now using \eqref{supI3final}, \eqref{I2sup}, \eqref{I1final} into \eqref{supestimate1} and letting $t_1\to \tau_1$, we obtain
\begin{equation}\label{seminorm2}
\begin{split}
&\frac{\zeta(\epsilon)}{C(p)}\int_{\tau_1}^{\tau_2+l}\int_{B_r}\int_{B_r}\big|\psi(x)u(x,t)^\frac{\alpha}{p}-\psi(y)u(y,t)^\frac{\alpha}{p}\big|^p \eta(t)\,d\mu\,dt\\
&\leq\Big(\zeta(\epsilon)+1+\frac{1}{\epsilon^{p-1}}\Big)\int_{\tau_1}^{\tau_2+l}\int_{B_r}\int_{B_r}|\psi(x)-\psi(y)|^p(u(x,t)^\alpha+u(y,t)^\alpha)\eta(t)\,d\mu\,dt\\
&\qquad+C(\Lambda)\sup_{x\in\text{ supp }\psi}\int_{\mathbb{R}^n\setminus B_r}\frac{dy}{|x-y|^{n+ps}}\int_{\tau_1}^{\tau_2+l}\int_{B_r}u(x,t)^\alpha\psi(x)^p\eta(t)\,dx\,dt\\
&\qquad+\frac{p-1}{\alpha}\int_{\tau_1}^{\tau_2+l}\int_{B_r}u(x,t)^\alpha\psi(x)^p|\partial_t\eta(t)|\,dx\,dt.
\end{split}
\end{equation}

We then choose $t_1$ such that
\begin{equation}\label{suptime}
\int_{B_r}u(x,t_1)^{p-\epsilon-1}\psi(x)^p\,dx\geq\frac{1}{2}\sup_{\tau_1<t<\tau_2}\int_{B_r}u(x,t)^{p-\epsilon-1}\psi(x)^p\,dx.
\end{equation}
Again using \eqref{supI3final}, \eqref{I2sup}, \eqref{I1final} and \eqref{suptime},  we get

\begin{equation}\label{parabolicnorm2}
\begin{split}
&\frac{p-1}{\alpha}\sup_{\tau_1<t<\tau_2}\int_{B_r}\psi(x)^p u(x,t)^\alpha\,dx\\
&\leq\Big(\zeta(\epsilon)+1+\frac{1}{\epsilon^{p-1}}\Big)\int_{\tau_1}^{\tau_2+l}\int_{B_r}\int_{B_r}|\psi(x)-\psi(y)|^p(u(x,t)^\alpha+u(y,t)^\alpha)\eta(t)\,d\mu\,dt\\
&\qquad+C(\Lambda)\sup_{x\in\text{ supp }\psi}\int_{\mathbb{R}^n\setminus B_r}\frac{dy}{|x-y|^{n+ps}}\int_{\tau_1}^{\tau_2+l}\int_{B_r}u(x,t)^\alpha\psi(x)^p\eta(t)\,dx\,dt\\
&\qquad+\frac{p-1}{\alpha}\int_{\tau_1}^{\tau_2+l}\int_{B_r}u(x,t)^\alpha\psi(x)^p|\partial_t\eta(t)|\,dx\,dt.
\end{split}
\end{equation}
Therefore from \eqref{seminorm2} and \eqref{parabolicnorm2} we get the required estimate.
\end{proof}

Following the energy estimate, we  now proceed with the proof of  the  reverse H$\ddot{\text{o}}$lder inequality for strictly positive supersolutions as in Theorem \ref{revholder}.

\medskip

\textbf{Proof of Theorem \ref{revholder}}
We will prove the Theorem when $sp<n$ with $\kappa=\frac{n+ps}{n}$. Similar argument holds in case when $sp\geq n$.
Let us denote by
$$
r_0=r,
\quad r_j=\Big(1-(1-\theta)\frac{1-2^{-j}}{(1-2^{-m})}\Big)r,
\quad \delta_j=2^{-j}r,\,j=1,2,\dots,m
$$
and $U_j=B_j\times\Gamma_j=B_{r_j}(x_0)\times (t_0,t_0+r_{j}^{ps})$.
We shall fix $m$ later.
Now we choose nonnegative test functions $\psi_j\in C_{c}^\infty(B_j)$ such that
$0\leq\psi_j\leq 1$ in $B_j$, $\psi_j\equiv 1$ in $B_{j+1}$,
$|\nabla\psi_j|\leq\frac{2^{j+3}}{(1-\theta)r}$
and $\text{dist}\,(\text{supp}\,\psi_j,\mathbb{R}^n\setminus B_j)\geq\frac{\delta_j(1-\theta)}{2}$.
Moreover we choose $\eta_j\in C^\infty(\Gamma_j)$ such that 
$0\leq\eta_j\leq 1$ in $\Gamma_j$, $\eta_j\equiv 1$ in $\Gamma_{j+1}$,
and $|\partial_t\eta_j|\leq\frac{2^{ps(j+3)}}{(1-\theta)r^{ps}}$, $\eta_j(t_0+r_j^{ps})=0$.
Let $\alpha=p-\epsilon-1$ where $\epsilon\in(0,p-1)$. Then $\alpha\in(0,p-1)$. Denote by $v=u^\frac{\alpha}{p}$. Let $r=r_j$, $\tau_1=t_0$, $\tau_2=t_0+r_{j+1}^{ps}$, $l=r_j^{ps}-r_{j+1}^{ps}$. By the Sobolev embedding theorem (Lemma \ref{parabolic embedding}), we obtain
\begin{equation}\label{parem1}
\begin{split}
\int_{\Gamma_{j+1}}\fint_{B_{j+1}}|v|^{p\kappa}\,dx\,dt&\leq C\Big(r_{j+1}^{ps-n} I_1+\int_{\Gamma_{j+1}}\fint_{B_{j+1}}|v|^p\,dx dt\Big)\cdot\Big(\frac{I_2}{|B_{j+1}|}\Big)^\frac{ps}{n},
\end{split}
\end{equation}
where
$$
I_1=\int_{\Gamma_{j+1}}\int_{B_{j+1}}\int_{B_{j+1}}\frac{|v(x,t)-v(y,t)|^p}{|x-y|^{n+sp}}\,dx\,dy\,dt,
$$
and 
$$
I_2=\sup_{\Gamma_{j+1}}\int_{B_{j+1}}|v|^p\,dx.
$$
Using the fact that $\psi_j \equiv 1$ on $B_{j+1}$ and also that $\eta_j \equiv 1$ on $\Gamma_{j+1}$, we obtain using  Lemma \ref{supenergy1} that the following holds,
\begin{equation}\label{I1rev}
I_1, I_2\leq C(J_1+J_2+J_3),
\end{equation}
for some positive constant $C$ which is  independent of $\alpha$ as long as $\alpha$ is away from $p-1$, where
\begin{equation}\label{paremJ1}
\begin{split}
J_1&=\int_{\Gamma_j}\int_{B_j}\int_{B_j}(v(x,t)^p+v(y,t)^p)\frac{|\psi_j(x)-\psi_j(y)|^p}{|x-y|^{n+sp}}\eta_j(t)\,dx\,dy\,dt\\
&\leq C\frac{2^{j(n+sp+p)}}{(1-\theta)^p r_j^{ps}}\int_{\Gamma_j}\int_{B_j}v(x,t)^p\,dx\,dt,
\end{split}
\end{equation}
since $r_j<r$,

\begin{equation}\label{paremJ2}
\begin{split}
J_2&=\sup_{x\in\text{supp}\,\psi_{j}}\int_{\mathbb{R}^n\setminus B_j}\frac{dy}{|x-y|^{n+sp}}\int_{\Gamma_j}\int_{B_j}v(x,t)^p\psi_j(x)^p\eta_j(t)\,dx\,dt\\
&\leq C\frac{2^{j(n+sp+p)}}{r_j^{ps}}\int_{\Gamma_j}\int_{B_j}v(x,t)^p\,dx\,dt,
\end{split}
\end{equation}
and
\begin{equation}\label{paremJ3}
J_3=\int_{\Gamma_j} \int_{B_j}\psi_j(x)^p v(x,t)^p|\partial_{t}\eta_j(t)|\,dx\,dt
\leq C\frac{2^{ps(j+3)}}{(1-\theta)r_j^{ps}}\int_{\Gamma_j}\int_{B_j}v(x,t)^{p}\,dx\,dt,
\end{equation}
again since $r_j < r$.

Therefore using \eqref{paremJ1}, \eqref{paremJ2} and \eqref{paremJ3} into \eqref{I1rev} we obtain
\begin{equation}\label{paremI1revfinal}
\begin{split}
I_1, I_2&\leq C\frac{2^{j(n+ps+p)}}{(1-\theta)^pr_j^{ps}}\int_{\Gamma_j}\int_{B_j}v(x,t)^p\,dx\,dt,
\end{split}
\end{equation}
for some positive constant $C$ independent of $\alpha$ as long as $\alpha$ is away from $p-1$, but may depend on $n,p,s,\Lambda$.

Using the estimate \eqref{paremI1revfinal} and the fact that $r_{j+1}<r_j<2 r_{j+1}$ for every $j$, we obtain from \eqref{parem1}, since $v=u^\frac{\alpha}{p}$ that
\begin{equation}\label{Moseritestep1}
\fint_{U_{j+1}}|u|^{\kappa\alpha}\, dx\, dt\leq C\Big(\frac{2^{j(n+ps+p)}}{(1-\theta)^{p}}\fint_{U_{j}}|u|^\alpha\,dx\, dt\Big)^\kappa,
\end{equation}
for some positive constant $C$ independent of $\alpha$ (given that our choice of $\alpha$ will be away from $p-1$) but may depend on $n,p,s,\Lambda$. Now we use Moser iteration technique into \eqref{Moseritestep1}. Since $\alpha\in(0,p-1)$, we have $\kappa\alpha\in(0,\kappa(p-1))$. This gives the upper bound $q_0=\kappa(p-1)$ for $q$. Let us fix $q,\bar{q}$ such that $0<\bar{q}<q$ and $m$ such that $\bar{q}\kappa^{m-1}\leq q\leq \bar{q}\kappa^{m}$. Let $t_0$ be such that $t_0\leq\bar{q}$ and $q=\kappa^{m}t_0$. Denote by $t_j=\kappa^{j}t_0$ for $j=0,1,\cdots,m$. Then observing that $r_m=\theta r$ and $r_0=r$, we get $U_m=U^{+}(\theta r)$ and $U_0=U^{+}(r)$. Hence 
\begin{align*}
\Big(\fint_{U^{+}{(\theta r)}}u^q\,dx\, dt\Big)^\frac{1}{q}&=\Big(\fint_{U_{m}}u^q\,dx\, dt\Big)^\frac{1}{q}\\
&\leq\Big(\frac{C 2^{\frac{(n+ps+p)m}{p}}}{(1-\theta)}\Big)^\frac{p}{t_{m-1}}\Big(\fint_{U_{m-1}}u^{t_{m-1}}\,dx\, dt\Big)^\frac{1}{t_{m-1}}\\
&\leq \Big(\frac{C_{\text{prod}}(m)}{(1-\theta)^{m^{*}}}\fint_{U^{+}(r)}u^{t_0}\,dx\, dt\Big)^\frac{1}{t_0},
\end{align*}
where
$$
C_{\text{prod}}(m)= C^{m^{*}}\prod_{j=0}^{m-1}\Big(2^{{\frac{n+ps+p}{p}}(j+1)}\Big)^{p\kappa^{-j}},
$$ 
and
$$
m^{*}=p\sum_{j=0}^{m-1}\kappa^{-j}=\frac{p\kappa}{\kappa-1}(1-\kappa^{-m}).
$$
It can be easily seen that $C_{\text{prod}}(m)$ is a positive constant uniformly bounded on $m$, where $C$ is independent of $\bar{q}$ but depends on $q$ due to the singularity of the constants involved   in the energy inequality in Lemma \ref{supenergy1} at $\epsilon=0$. Finally using H$\ddot{\text{o}}$lder's inequality, we obtain
\begin{equation*}\label{finalestimatewithr}
\begin{gathered}
\Big(\fint_{U^{+}(\theta r)}u^{q}\,dx\, dt\Big)^\frac{1}{q}\leq\Big(\frac{C}{(1-\theta)^{m^{*}}}\Big)^\frac{1}{t_0}\Big(\fint_{U^{+}(r)}u^{\bar{q}}\,dx\, dt\Big)^\frac{1}{\bar{q}}.
\end{gathered}
\end{equation*}
Now, since $\bar{q}\kappa^{m-1}\leq t_0\kappa^{m}$, we have $t_0\geq\frac{\bar{q}}{\kappa}$. As a consequence the required estimate follows with $\mu=\frac{p\kappa^2}{\kappa-1}$.

In closing,  we  prove the following logarithmic estimate for stricly positive supersolutions which constitutes the nonlocal analogue of Lemma 6.1 in \cite{Kin-Kuusi} and also constitutes one of the key ingredients in the proof of weak harnack in the local case. 
\begin{Lemma}\label{Logestimatelemma}
Let $p>2$, $x_0\in\mathbb{R}^n$, $r>0$ and $t_0\in(r^{ps},T-r^{ps})$. Suppose that $u$ is a weak supersolution to \eqref{maineqn} with $u\geq\rho>0$ in $\mathbb{R}^n\times (t_0-r^{ps},t_0+r^{ps})$. Then there exists a positive constant $C=C(n,p,s,\Lambda)$ such that
\begin{equation}\label{logest1}
\big|\{(x,t)\in U^{+}(x_0,t_0,r):\log u(x,t)<-\lambda-b\}\big|\leq \frac{Cr^{n+ps}}{\lambda^{p-1}}
\end{equation}
and
\begin{equation}\label{logest2}
\big|\{(x,t)\in U^{-}(x_0,t_0,r):\log u(x,t)>\lambda-b\}\big|\leq \frac{Cr^{n+ps}}{\lambda^{p-1}}
\end{equation}
where 
\[
b=b(u(\cdot,t_0))= - \frac{\int_{B_{\frac{3r}{2}}(x_0)} \log u(x,t_0)\psi(x)^p\,dx}{\int_{B_{\frac{3r}{2}}(x_0)}\psi(x)^{p}\,dx}.
\]
\end{Lemma}

\begin{proof}
Following Lemma 6.1 in \cite{Kin-Kuusi}, we only prove \eqref{logest1} because the proof of \eqref{logest2} is analogous. Without loss of generality, we may assume $x_0=0$. Let $\psi\in C_0^{\infty}(B_{\frac{3r}{2}})$ be a
nonnegative radially decreasing function  such that $0\leq\psi\leq 1$ in $B_\frac{3r}{2}$, $\psi\equiv 1$ \text{ in }$B_r$, $|\nabla\psi|\leq\frac{C}{r}$ in $B_\frac{3r}{2}$.  Since $u$ is a strictly positive supersolution of \eqref{maineqn}, choosing $\phi(x,t)=\psi(x)^p u(x,t)^{1-p}$ as a test function in \eqref{weaksubsupsoln}, we get
\begin{equation}\label{log1}
I_1+I_2+2I_3\geq 0,
\end{equation}
where for any $t_0-r^{ps}\leq t_1<t_2\leq t_0+r^{ps}$, we have
\begin{equation}\label{logI1}
I_1=\int_{t_1}^{t_2}\int_{B_{\frac{3r}{2}}}\frac{\partial}{\partial t}(u(x,t)^{p-1})\phi(x,t)\,dx\, dt
=(p-1)\int_{B_{\frac{3r}{2}}}\log\,u(x,t)\psi(x)^p\,dx\Big|_{t=t_1}^{t_2},
\end{equation}
\[
I_2=\int_{t_1}^{t_2}\int_{B_{\frac{3r}{2}}}\int_{B_{\frac{3r}{2}}}\mathcal{A}(u(x,y,t))(\phi(x,t)-\phi(y,t))\,d\mu\, dt,
\]
and
\[
I_3=\int_{t_1}^{t_2}\int_{\mathbb{R}^n\setminus B_{\frac{3r}{2}}}\int_{B_{\frac{3r}{2}}}\mathcal{A}(u(x,y,t))\phi(x,t)\,d\mu\, dt.
\]
Following the arguments in the proof of \cite[Lemma 1.3]{Kuusilocal}, we obtain for some positive constant $C=C(p)$,
\begin{equation}\label{logI2}
\begin{split}
I_2&=\int_{t_1}^{t_2}\int_{B_{\frac{3r}{2}}}\int_{B_{\frac{3r}{2}}}\mathcal{A}(u(x,y,t))(\phi(x,t)-\phi(y,t))\,d\mu\, dt\\
&\leq -\frac{1}{C}\int_{t_1}^{t_2}\int_{B_{\frac{3r}{2}}}\int_{B_{\frac{3r}{2}}}K(x,y,t)|\log\,u(x,t)-\log\,u(y,t)|^p \psi(y)^p\,dx\, dy\, dt\\
&\qquad+C\int_{t_1}^{t_2}\int_{B_{\frac{3r}{2}}}\int_{B_{\frac{3r}{2}}}K(x,y,t)|\psi(x)-\psi(y)|^p\,dx\, dy\, dt\\
&\leq -\frac{1}{C}\int_{t_1}^{t_2}\int_{B_{\frac{3r}{2}}}\int_{B_{\frac{3r}{2}}}K(x,y,t)|\log\,u(x,t)-\log\,u(y,t)|^p \psi(y)^p\,dx\, dy\, dt\\
&\qquad+C(t_2-t_1)r^{n-sp}, 
\end{split}
\end{equation}
where the last inequality is obtained using the properties of $\psi$. Again following the proof of \cite[Lemma 1.3]{Kuusilocal}, we get that 
\begin{equation}\label{logI3}
I_3=\int_{t_1}^{t_2}\int_{\mathbb{R}^n\setminus B_{\frac{3r}{2}}}\int_{B_{\frac{3r}{2}}}\mathcal{A}(u(x,y,t))\phi(x,t)\,d\mu\, dt
\leq C(t_2-t_1)r^{n-sp}.
\end{equation}

Therefore using the estimates \eqref{logI1}, \eqref{logI2} and \eqref{logI3} into \eqref{log1}, we obtain
\begin{equation}\label{log2}
\begin{gathered}
\frac{1}{C}\int_{t_1}^{t_2}\int_{B_{\frac{3r}{2}}}\int_{B_{\frac{3r}{2}}}K(x,y,t)|\log\,u(x,t)-\log\,u(y,t)|^p \psi(y)^p\,dx\, dy\, dt\\
-(p-1)\int_{B_{\frac{3r}{2}}}\log\,u(x,t)\psi(x)^p\,dx\Big|_{t=t_1}^{t_2}\leq C(t_2-t_1)r^{n-sp}.
\end{gathered}
\end{equation}
Let $v(x,t)=-\log\,u(x,t)$ and
$$
V(t)=\frac{\int_{B_{\frac{3r}{2}}}v(x,t)\psi(x)^p\,dx}{\int_{B_{\frac{3r}{2}}}\psi(x)^{p}\,dx}.
$$
Since $0\leq\psi\leq 1$ in $B_{\frac{3r}{2}}$ and $\psi\equiv 1$ in $B_r$, therefore we have that $\int_{B_{\frac{3r}{2}}}\psi(x)^p\,dx\approx r^n$. Hence dividing by $\int_{B_{\frac{3r}{2}}}\psi(x)^p\,dx$ on both sides of \eqref{log2}, we obtain using the weighted Poincar\'{e} inequality in Lemma \ref{wgtPoincare} that the following holds,
$$
V(t_2)-V(t_1)+\frac{r^{-sp}}{c(p-1)}\int_{t_1}^{t_2}\fint_{B_r}|v(x,t)-V(t)|^p\,dx\, dt\leq\frac{cr^{-sp}}{p-1}(t_2-t_1).
$$
Let $A_1=C(p-1)$, $A_2=\frac{C}{p-1}$, 
\[
\bar{w}(x,t)=v(x,t)-A_2 r^{-sp}(t-t_1)
\quad\text{and}\quad
\bar{W}(t)=V(t)-A_2 r^{-sp}(t-t_1).
\]
Therefore $v(x,t)-V(t)=\bar{w}(x,t)-\bar{W}(t)$.
Hence we get
\begin{equation}\label{monotone1}
\bar{W}(t_2)-\bar{W}(t_1)+\frac{1}{A_1 r^{n+ps}}\int_{t_1}^{t_2}\int_{B_r}|\bar{w}(x,t)-{\bar{W}(t)}|^p\, dx\, dt\leq 0.
\end{equation}
Therefore $\bar{W}(t)$ is a monotone decreasing function in $t_0-r^{ps}\leq t_1<t_2\leq t_0+r^{ps}$. Hence  $\bar{W}(t)$ is differentiable almost everywhere with respect to $t$. 
Dividing by $t_2-t_1$ on both sides of \eqref{monotone1}, we obtain after letting $t_2\to t_1$, 
\begin{equation}\label{notime}
\bar{W}'(t)+\frac{1}{A_1 r^{n+ps}}\int_{B_r}\big|\bar{w}(x,t)-\bar{W}(t)\big|^p\,dx\leq 0.
\end{equation}

Let $t_1=t_0$, then $\bar{W}(t_0)=V(t_0)$ and we denote by $b(u(\cdot,t_0))= \bar{W}(t_0)$. 
Let
$$
\Omega_{t}^{+}(\lambda)=\big\{x\in B_r:\bar{w}(x,t)>b+\lambda\big\}.
$$
Then for any $x\in\Omega_{t}^{+}(\lambda)$ and $t\geq t_0$, since $\bar{W}(t)\leq \bar{W}(t_0)=b$, we have 
\[
\bar{w}(t,x)-\bar{W}(t)\geq b+\lambda-\bar{W}(t)\geq b+\lambda-\bar{W}(t_0)=\lambda>0.
\]
Hence from \eqref{notime}, we have 
$$
\bar{W}'(t)+\frac{|\Omega_{t}^{+}(\lambda)|}{A_1 r^{n+sp}}\big(b+\lambda-\bar{W}(t)\big)^p\leq 0.
$$
Therefore, we have
$$
|\Omega_{t}^{+}(\lambda)|\leq -\frac{A_1 r^{n+sp}}{p-1}\partial_{t}\big(b+\lambda-\bar{W}(t)\big)^{1-p}.
$$
Integrating over $t_0$ to $t_0+r^{sp}$, we obtain
$$
\big|\{(x,t)\in B_r\times(t_0,t_0+r^{sp}):\bar{w}(x,t)>b+\lambda\}\big|\leq-\frac{A_1 r^{n+sp}}{p-1}\int_{t_0}^{t_0+r^{sp}}\partial_{t}\big(b+\lambda-\bar{W}(t)\big)^{1-p}\,dt,
$$
which gives
\begin{equation}\label{logsubpositivetime}
\big|\{(x,t)\in B_r\times(t_0,t_0+r^{sp}):\log\,u(x,t)+A_2 r^{-sp}(t-t_0)<-\lambda-b\}\big|\leq\frac{A_1}{p-1}\frac{r^{n+sp}}{\lambda^{p-1}}.
\end{equation}
Finally we  note that,
\begin{equation}\label{logpostivetime}
\big|\{(x,t)\in B_r\times(t_0,t_0+r^{sp}):\log\,u(x,t)<-\lambda-b\}\big|\leq A+B,
\end{equation}
where 
$$
A=\big|\{(x,t)\in B_r\times(t_0,t_0+r^{sp}):\log\,u(x,t)+A_2 r^{-sp}(t-t_0)<-\tfrac{\lambda}{2}-b\}\big|
\leq \frac{Cr^{n+sp}}{\lambda^{p-1}},
$$
which follows from \eqref{logsubpositivetime} 
and
$$
B=\big|\{(x,t)\in B_r\times(t_0,t_0+r^{sp}):A_2 r^{-sp}(t-t_0)>\tfrac{\lambda}{2}\}\big|
\leq \Big(1-\frac{\lambda}{2A_2}\Big)r^{n+sp}.  
$$
If $\frac{\lambda}{2 A_2}<1$, then 
$$B\leq\Big(1-\frac{\lambda}{2A_2}\Big)r^{n+sp}<r^{n+sp}<\Big(\frac{2 A_2}{\lambda}\Big)^{p-1}r^{n+sp}.$$
If $\frac{\lambda}{2 A_2}\geq 1$, then $B = 0$. Hence in either case we have
$$
B\leq\frac{Cr^{n+ps}}{\lambda^{p-1}}.
$$
Inserting the above estimates of $A$ and $B$ into \eqref{logpostivetime}, we obtain
\begin{equation*}\label{logpostivetimefinal}
\big|\{(x,t)\in B_r\times(t_0,t_0+r^{sp}):\log\,u(x,t)<-\lambda-b\}\big|\leq\frac{Cr^{n+ps}}{\lambda^{p-1}},
\end{equation*}
for some positive constant $C=C(n,p,s,\Lambda)$, which proves \eqref{logest1}. The proof of \eqref{logest2} is analogous.

\end{proof}

\section{Appendix}
In this section, we prove Lemma \ref{Inequality1}.  To this end, we establish the following auxiliary lemmas. Throughout this section, we assume $p>1$. 
\begin{Lemma}\label{auxlemma1}
Let $f,g\in C^1([a,b])$. Then 
\begin{equation*}
\frac{f(b)-f(a)}{b-a}+\Big|\frac{g(b)-g(a)}{b-a}\Big|^p\leq\max_{t\in[a,b]}\big[f'(t)+|g'(t)|^p\big].
\end{equation*}
\end{Lemma}
\begin{proof}
Suppose the result does not hold, then by contradiction, we get
\begin{equation*}
\frac{f(b)-f(a)}{b-a}+\Big|\frac{g(b)-g(a)}{b-a}\Big|^p>f'(t)+|g'(t)|^p,
\end{equation*} 
for all $t\in[a,b]$. Integrating over $a$ to $b$, we obtain
\begin{equation*}
\Big|\frac{g(b)-g(a)}{b-a}\Big|^p>\frac{1}{b-a}\int_{a}^{b}|g'(t)|^p\,dt,
\end{equation*}
which contradicts Jensen's inequality.
\end{proof}
\begin{Lemma}\label{auxlemma2}
Let $a,b>0$, $0<\epsilon<p-1$. Then we have
\begin{equation*}
|b-a|^{p-2}(b-a)(a^{-\epsilon}-b^{-\epsilon})\geq\zeta(\epsilon)\Big|b^\frac{p-\epsilon-1}{p}-a^\frac{p-\epsilon-1}{p}\Big|^p,
\end{equation*}
where $\zeta(\epsilon)=\frac{p^p\epsilon}{(p-\epsilon-1)^p}$. Moreover, if $0<p-\epsilon-1<1$, then we may choose $\zeta(\epsilon)=\frac{p^p\epsilon}{p-\epsilon-1}$.
\end{Lemma}
\begin{proof}
Let $0<\epsilon<p-1$ and $\zeta(\epsilon)=\frac{p^p\epsilon}{(p-\epsilon-1)^p}$. Let $f(t)=\frac{t^{-\epsilon}}{\zeta(\epsilon)}$ and $g(t)=t^\frac{p-\epsilon-1}{p}$. By Lemma \ref{auxlemma1}, we have
$$
\frac{1}{\zeta(\epsilon)}\frac{b^{-\epsilon}-a^{-\epsilon}}{b-a}+\Big|\frac{b^\frac{p-\epsilon-1}{p}-a^\frac{p-\epsilon-1}{p}}{b-a}\Big|^p\leq 0.
$$
If $b\geq a$, multiplying by $(b-a)^p$, we obtain
\begin{equation}\label{epsilonless}
(b-a)^{p-1}(a^{-\epsilon}-b^{-\epsilon})\geq\zeta(\epsilon)\big|b^\frac{p-\epsilon-1}{p}-a^\frac{p-\epsilon-1}{p}\big|^p.
\end{equation}
If $b<a$, interchanging $a$ and $b$, the Lemma follows. If $0<p-\epsilon-1<1$, then we have $0<(p-\epsilon-1)^p<p-\epsilon-1$, we have $\zeta(\epsilon)\geq\frac{p^p\epsilon}{p-\epsilon-1}$
and \eqref{epsilonless} implies
\begin{equation*}
(b-a)^{p-1}(a^{-\epsilon}-b^{-\epsilon})\geq\frac{p^p\epsilon}{p-\epsilon-1}\Big|b^\frac{p-\epsilon-1}{p}-a^\frac{p-\epsilon-1}{p}\Big|^p.
\end{equation*}
Hence the claim follows with $\zeta(\epsilon)=\frac{p^p\epsilon}{p-\epsilon-1}$ when $0<p-\epsilon-1<1$.
\end{proof}

\subsection{Proof of Lemma \ref{Inequality1}}
We denote the left-hand and right-hand sides of \eqref{Ineeqn} by L.H.S and R.H.S, respectively.
Let $\zeta_1(\epsilon)=\frac{\zeta(\epsilon)}{c(p)}$ and $\zeta_2(\epsilon)=\zeta(\epsilon)+1+\frac{1}{\epsilon^{p-1}}$. Then $\zeta_1(\epsilon)-\zeta_2(\epsilon)<-1$ since $C(p)>1$ (to be finally chosen appropriately).

Case 1. If $\tau_1=\tau_2=0$, then \eqref{Ineeqn} holds trivially.

Case 2. If $\tau_1>0$ and $\tau_2=0$. 
In this case, we note that if  $b>a$, then
$$
\text{L.H.S}=|b-a|^{p-2}(b-a)(\tau_1^{p}a^{-\epsilon}-\tau_2^{p}b^{-\epsilon})=(b-a)^{p-1}\tau_1^{p}a^{-\epsilon}
$$
and
\begin{align*}
\text{R.H.S}&=\zeta_1(\epsilon)\tau_1^{p}a^{p-\epsilon-1}-\zeta_2(\epsilon)\tau_1^{p}(b^{p-\epsilon-1}+a^{p-\epsilon-1})\\
&=(\zeta_1(\epsilon)-\zeta_2(\epsilon))\tau_1^{p}a^{p-\epsilon-1}-\zeta_2(\epsilon)\tau_1^{p}b^{p-\epsilon-1}.
\end{align*}
Now L.H.S is positive and since $\zeta_1(\epsilon)-\zeta_2(\epsilon)<0$ and $\zeta_2(\epsilon)>0$, the R.H.S is negative. Therefore we have L.H.S $\geq$ R.H.S.
On the other hand if $b\leq a$, then 
\begin{align*}
\text{L.H.S}&=-(a-b)^{p-1}\tau_1^{p}a^{-\epsilon}\geq -\tau_1^{p}a^{p-\epsilon-1},
\end{align*}
and since $\zeta_1(\epsilon)-\zeta_2(\epsilon)<-1$ and $\zeta_2(\epsilon)>0$, we have
$$
\text{R.H.S}=(\zeta_1(\epsilon)-\zeta_2(\epsilon))\tau_1^{p}a^{p-\epsilon-1}-\zeta_2(\epsilon)\tau_1^{p}b^{p-\epsilon-1}
<-\tau_1^{p}a^{p-\epsilon-1}
\leq\text{L.H.S}.
$$

Case 3. If $\tau_1=0$ and $\tau_2>0$. Then we have
$$
\text{L.H.S}=-|b-a|^{p-2}(b-a)\tau_2^{p}b^{-\epsilon},
$$
and 
\begin{align*}
\text{R.H.S}&=(\zeta_1(\epsilon)-\zeta_2(\epsilon))\tau_2^{p}b^{p-\epsilon-1}-\zeta_2(\epsilon)\tau_2^{p}a^{p-\epsilon-1}.
\end{align*}

If $b>a$, then
$$
\text{L.H.S}=-(b-a)^{p-1}\tau_2^{p}b^{-\epsilon}\geq-\tau_2^{p}b^{p-\epsilon-1},
$$
and since $\zeta_1(\epsilon)-\zeta_2(\epsilon)<-1$ and $\zeta_2(\epsilon)>0$, we have 
$$
\text{R.H.S}=(\zeta_1(\epsilon)-\zeta_2(\epsilon))\tau_2^{p}b^{p-\epsilon-1}-\zeta_2(\epsilon)\tau_2^{p}a^{p-\epsilon-1}\\
<-\tau_2^{p}b^{p-\epsilon-1}
\leq\text{L.H.S}.
$$

If $b\leq a$, then the L.H.S is nonnegative and the R.H.S is negative. Therefore we have $\text{L.H.S}\geq\text{R.H.S}$.

Case 4. Let both $\tau_1,\tau_2>0$. By symmetry, we may assume that $b\geq a$. Let $t=\frac{b}{a}\geq 1$, $s=\frac{\tau_2}{\tau_1}>0$ and $\lambda=s^{p}t^{-\epsilon}$. It can be easily seen that the inequality \eqref{Ineeqn} is equivalent to the following inequality
\begin{equation}\label{equivineq}
\zeta_1(\epsilon)|s t^\frac{p-\epsilon-1}{p}-1|^p\leq (t-1)^{p-1}(1-\lambda)+\zeta_2(\epsilon)|s-1|^p(t^{p-\epsilon-1}+1).
\end{equation}
We first estimate the following term.
\begin{align*}
|st^\frac{p-\epsilon-1}{p}-1|^p&=|st^\frac{p-\epsilon-1}{p}-t^\frac{p-\epsilon-1}{p}+t^\frac{p-\epsilon-1}{p}-1|^p\\
&=|(s-1)t^\frac{p-\epsilon-1}{p}+(t^\frac{p-\epsilon-1}{p}-1)|^p\\
&\leq 2^{p-1}|s-1|^p t^{p-\epsilon-1}+2^{p-1}|t^\frac{p-\epsilon-1}{p}-1|^p\\
&=A+B,
\end{align*}
where 
$$
A=2^{p-1}|s-1|^p t^{p-\epsilon-1}\text{ and }B=2^{p-1}|t^\frac{p-\epsilon-1}{p}-1|^p.
$$
By Lemma \ref{auxlemma2}, we have
$$
B\leq \frac{2^{p-1}(t-1)^{p-1}(1-t^{-\epsilon})}{\zeta(\epsilon)}.
$$
As a consequence, we obtain 
$$
|st^\frac{p-\epsilon-1}{p}-1|^p\leq 2^{p-1}|s-1|^p t^{p-\epsilon-1}+\frac{2^{p-1}(t-1)^{p-1}(1-t^{-\epsilon})}{\zeta(\epsilon)}.
$$
We observe that
\begin{align*}
1-t^{-\epsilon}&=1-\lambda+\lambda-t^{-\epsilon}
=1-\lambda+(s^p-1)t^{-\epsilon}\\
&=1-\lambda+|s-1|^p t^{-\epsilon}+(s^p-1-|s-1|^p)t^{-\epsilon}.
\end{align*}
Therefore, we get
\begin{equation}\label{auxestimateineq1}
\begin{split}
&|s t^\frac{p-\epsilon-1}{p}-1|^p
\leq2^{p-1}\big(1+\frac{1}{\zeta(\epsilon)}\big)|s-1|^p t^{p-\epsilon-1}\\
&\qquad+\frac{2^{p-1}}{\zeta(\epsilon)}(t-1)^{p-1}(1-\lambda)+\frac{2^{p-1}}{\zeta(\epsilon)}(t-1)^{p-1}(s^p-1-|s-1|^p)t^{-\epsilon}.
\end{split}
\end{equation}

Next we estimate the term $T=\frac{2^{p-1}}{\zeta(\epsilon)}(t-1)^{p-1}(s^p-1-|s-1|^p)t^{-\epsilon}$ for different values of $t$ and $s$. 

Case (a). If $t>1$ and $s\geq 2$. Then using the fact that $s\geq 2$, it can be easily seen that there exists constant $C(p)$ large enough such that $s^p-1-(s-1)^p\leq C(p)(s-1)^p$. Therefore we get
\begin{equation}\label{case1C}
T\leq\frac{C(p)}{\zeta(\epsilon)}|s-1|^p t^{p-\epsilon-1}.
\end{equation}
By inserting \eqref{case1C} into \eqref{auxestimateineq1}, we get
\begin{equation}\label{auxestimateineq1a}
\begin{gathered}
|s t^\frac{p-\epsilon-1}{p}-1|^p\leq C(p)\big(1+\frac{1}{\zeta(\epsilon)}\big)|s-1|^p t^{p-\epsilon-1}+\frac{C(p)}{\zeta(\epsilon)}(t-1)^{p-1}(1-\lambda).
\end{gathered}
\end{equation}

Case (b). If $t=1$ or $0<s\leq 1$. Then $T\leq 0$. 
Hence we get the estimate in \eqref{auxestimateineq1a}.

Case (c). If $t>1$, $s\in(1,2)$. Let $r\geq p$ be the nearest integer to $p$. Again it follows that there exists a positive constant $C(p)$ large enough such that  $s^p-1-|s-1|^p\leq C(p)|s-1|$.  We have further subcases.

Case (i). If $$t-1<\frac{r 2^{r-1}}{\epsilon}t(s-1).$$  
Note that  we can choose $C(p)$ large enough such that $r 2^{r-1}\leq C(p)$. Hence we have
\begin{equation}\label{case3C}
T\leq \frac{C(p)}{\epsilon^{p-1}\zeta(\epsilon)}t^{p-\epsilon-1}|s-1|^p.
\end{equation}
By inserting \eqref{case3C} into \eqref{auxestimateineq1}, we get
\begin{equation}\label{auxestimateineq1c}
|s t^\frac{p-\epsilon-1}{p}-1|^p
\leq C(p)\Big(1+\frac{1}{\zeta(\epsilon)}\Big(1+\frac{1}{\epsilon^{p-1}}\Big)\Big)|s-1|^p t^{p-\epsilon-1}+\frac{C(p)}{\zeta(\epsilon)}(t-1)^{p-1}(1-\lambda).
\end{equation}

Case (ii). If $$t-1\geq\frac{r 2^{r-1}}{\epsilon}t(s-1).$$
Since $r$ is an integer, we observe that
$$
s^r+s-2=(s-1)(s^{r-1}+s^{r-2}+\cdots+s+2).
$$ 
By the mean value theorem there exists $\eta\in(1,t)$ such that $t^\epsilon-1=\epsilon\eta^{\epsilon-1}(t-1)$ and so $\epsilon=\frac{t^\epsilon-1}{\eta^{\epsilon-1}(t-1)}$.
Now, we have
\begin{align*}
\frac{s^r+s-2}{t-1}
&=\frac{s-1}{t-1}(s^{r-1}+s^{r-2}\cdots+s+2)\\
&\leq\frac{\epsilon}{r 2^{r-1}t}(s^{r-1}+s^{r-2}+\cdots+s+2)\\
&\leq\frac{\epsilon}{t}
\leq\frac{t^\epsilon-1}{t\eta^{\epsilon-1}(t-1)},
\end{align*}
which gives $t\eta^{\epsilon-1}(s^r+s-2)\leq t^\epsilon-1$.

Now, the fact $\epsilon>0$ and $1<\eta<t$ gives $t\eta^{\epsilon-1}>\eta^\epsilon>1$. Therefore since $r\geq p$ and $s>1$, we get $s^p+s-2<s^r+s-2<t\eta^{\epsilon-1}(s^r+s-2)\leq t^\epsilon-1$. Hence we have $s-1\leq t^\epsilon-s^p=t^\epsilon(1-\lambda)$. Thus
\begin{equation}\label{case4C}
T\leq\frac{C(p)}{\zeta(\epsilon)}(t-1)^{p-1}(1-\lambda).
\end{equation}
Using \eqref{case4C} into \eqref{auxestimateineq1} we get
\begin{equation}\label{auxestimateineq1d}
|s t^\frac{p-\epsilon-1}{p}-1|^p\leq C(p)\big(1+\frac{1}{\zeta(\epsilon)}\big)|s-1|^p t^{p-\epsilon-1}+\frac{C(p)}{\zeta(\epsilon)}(t-1)^{p-1}(1-\lambda).
\end{equation}

Finally from the estimates \eqref{auxestimateineq1a}, \eqref{auxestimateineq1c} and \eqref{auxestimateineq1d}, we obtain
\begin{equation}\label{finalestimateineq}
\begin{split}
|st^\frac{p-\epsilon-1}{p}-1|^p\leq C(p)\Big(1+\frac{1}{\zeta(\epsilon)}(1+\frac{1}{\epsilon^{p-1}})\Big)|s-1|^p (t^{p-\epsilon-1}+1)\\+\frac{C(p)}{\zeta(\epsilon)}(t-1)^{p-1}(1-\lambda).
\end{split}
\end{equation}
Multiplying $\frac{\zeta(\epsilon)}{C(p)}$ on both sides of \eqref{finalestimateineq}, we obtain
\begin{equation*}\label{finalestimateineq1}
\frac{\zeta(\epsilon)}{C(p)}|st^\frac{p-\epsilon-1}{p}-1|^p\leq\Big(\zeta(\epsilon)+1+\frac{1}{\epsilon^{p-1}}\Big)|s-1|^p(t^{p-\epsilon-1}+1)+(t-1)^{p-1}(1-\lambda),
\end{equation*}
which corresponds to the inequality \eqref{equivineq}. The lemma  thus follows.

\section*{Acknowledgements}
{A.B. is supported in part by SERB Matrix grant MTR/2018/000267 and by Department of Atomic Energy,  Government of India, under
project no.  12-R \& D-TFR-5.01-0520. P.G. and J.K. are supported by the Academy of Finland.}

\medskip
\noindent
Agnid Banerjee\\
Tata Institute of Fundamental Research\\
Centre For Applicable Mathematics\\
Bangalore-560065, India\\
Email: agnidban@gmail.com\\

\noindent
Prashanta Garain\\
Department of Mathematics\\
P.O. Box 11100\\
FI-00076 Aalto University, Finland\\
Email: pgarain92@gmail.com\\

\noindent
Juha Kinnunen\\
Department of Mathematics\\
P.O. Box 11100\\
FI-00076 Aalto University, Finland\\
Email: juha.k.kinnunen@aalto.fi


\begin{thebibliography}{10}

\bibitem{Kinhigherint}
Verena {B{\"o}gelein}, Frank {Duzaar}, Juha {Kinnunen}, and Christoph
  {Scheven}.
\newblock {Higher integrability for doubly nonlinear parabolic systems}.
\newblock {\em arXiv e-prints}, page arXiv:1810.06039, Oct 2018.

\bibitem{Verenacontinuity}
Verena {B{\"o}gelein}, Frank {Duzaar}, and Naian {Liao}.
\newblock {On the H{\"o}lder regularity of signed solutions to a doubly
  nonlinear equation}.
\newblock {\em arXiv e-prints}, page arXiv:2003.04158, March 2020.

\bibitem{Bo}
E.~Bombieri and E.~Giusti.
\newblock Harnack's inequality for elliptic differential equations on minimal
  surfaces.
\newblock {\em Invent. Math.}, 15:24--46, 1972.

\bibitem{Vazquez}
Matteo Bonforte, Yannick Sire, and Juan~Luis V\'{a}zquez.
\newblock Optimal existence and uniqueness theory for the fractional heat
  equation.
\newblock {\em Nonlinear Anal.}, 153:142--168, 2017.

\bibitem{Brascolind}
Lorenzo Brasco and Erik Lindgren.
\newblock Higher {S}obolev regularity for the fractional {$p$}-{L}aplace
  equation in the superquadratic case.
\newblock {\em Adv. Math.}, 304:300--354, 2017.

\bibitem{Braslinsck}
Lorenzo Brasco, Erik Lindgren, and Armin Schikorra.
\newblock Higher {H}\"{o}lder regularity for the fractional {$p$}-{L}aplacian
  in the superquadratic case.
\newblock {\em Adv. Math.}, 338:782--846, 2018.

\bibitem{Martincont}
Lorenzo {Brasco}, Erik {Lindgren}, and Martin {Str{\"o}mqvist}.
\newblock {Continuity of solutions to a nonlinear fractional diffusion
  equation}.
\newblock {\em arXiv e-prints}, page arXiv:1907.00910, Jul 2019.

\bibitem{Braspar}
Lorenzo Brasco and Enea Parini.
\newblock The second eigenvalue of the fractional {$p$}-{L}aplacian.
\newblock {\em Adv. Calc. Var.}, 9(4):323--355, 2016.

\bibitem{Cafchanvas}
Luis Caffarelli, Chi~Hin Chan, and Alexis Vasseur.
\newblock Regularity theory for parabolic nonlinear integral operators.
\newblock {\em J. Amer. Math. Soc.}, 24(3):849--869, 2011.

\bibitem{Kassmanchaker}
Jamil Chaker and Moritz Kassmann.
\newblock Nonlocal operators with singular anisotropic kernels.
\newblock {\em Comm. Partial Differential Equations}, 45(1):1--31, 2020.

\bibitem{Cozzi}
Matteo Cozzi.
\newblock Fractional {D}e {G}iorgi classes and applications to nonlocal
  regularity theory.
\newblock In {\em Contemporary research in elliptic {PDE}s and related topics},
  volume~33 of {\em Springer INdAM Ser.}, pages 277--299. Springer, Cham, 2019.

\bibitem{Kuusiharnack}
Agnese Di~Castro, Tuomo Kuusi, and Giampiero Palatucci.
\newblock Nonlocal {H}arnack inequalities.
\newblock {\em J. Funct. Anal.}, 267(6):1807--1836, 2014.

\bibitem{Kuusilocal}
Agnese Di~Castro, Tuomo Kuusi, and Giampiero Palatucci.
\newblock Local behavior of fractional {$p$}-minimizers.
\newblock {\em Ann. Inst. H. Poincar\'{e} Anal. Non Lin\'{e}aire},
  33(5):1279--1299, 2016.

\bibitem{Hitchhiker'sguide}
Eleonora Di~Nezza, Giampiero Palatucci, and Enrico Valdinoci.
\newblock Hitchhiker's guide to the fractional {S}obolev spaces.
\newblock {\em Bull. Sci. Math.}, 136(5):521--573, 2012.

\bibitem{Dibe}
Emmanuele DiBenedetto.
\newblock {\em Degenerate parabolic equations}.
\newblock Universitext. Springer-Verlag, New York, 1993.

\bibitem{Kassmaninequality}
Bart\l~omiej Dyda and Moritz Kassmann.
\newblock On weighted {P}oincar\'{e} inequalities.
\newblock {\em Ann. Acad. Sci. Fenn. Math.}, 38(2):721--726, 2013.

\bibitem{Kassweakharnack}
Matthieu Felsinger and Moritz Kassmann.
\newblock Local regularity for parabolic nonlocal operators.
\newblock {\em Comm. Partial Differential Equations}, 38(9):1539--1573, 2013.

\bibitem{Xavier}
Xavier Fern\'{a}ndez-Real and Xavier Ros-Oton.
\newblock Boundary regularity for the fractional heat equation.
\newblock {\em Rev. R. Acad. Cienc. Exactas F\'{\i}s. Nat. Ser. A Mat. RACSAM},
  110(1):49--64, 2016.

\bibitem{Gvespri}
Ugo Gianazza and Vincenzo Vespri.
\newblock A {H}arnack inequality for solutions of doubly nonlinear parabolic
  equations.
\newblock {\em J. Appl. Funct. Anal.}, 1(3):271--284, 2006.

\bibitem{LinHynd1}
Ryan Hynd and Erik Lindgren.
\newblock A doubly nonlinear evolution for the optimal {P}oincar\'{e}
  inequality.
\newblock {\em Calc. Var. Partial Differential Equations}, 55(4):Art. 100, 22,
  2016.

\bibitem{LinHynd}
Ryan Hynd and Erik Lindgren.
\newblock H\"{o}lder estimates and large time behavior for a nonlocal doubly
  nonlinear evolution.
\newblock {\em Anal. PDE}, 9(6):1447--1482, 2016.

\bibitem{LinHynd2}
Ryan Hynd and Erik Lindgren.
\newblock Lipschitz regularity for a homogeneous doubly nonlinear {PDE}.
\newblock {\em SIAM J. Math. Anal.}, 51(4):3606--3624, 2019.

\bibitem{KassmanHarnack}
Moritz Kassmann.
\newblock A new formulation of {H}arnack's inequality for nonlocal operators.
\newblock {\em C. R. Math. Acad. Sci. Paris}, 349(11-12):637--640, 2011.

\bibitem{KassmanSch}
Moritz Kassmann and Russell~W. Schwab.
\newblock Regularity results for nonlocal parabolic equations.
\newblock {\em Riv. Math. Univ. Parma (N.S.)}, 5(1):183--212, 2014.

\bibitem{Kim}
Yong-Cheol Kim.
\newblock Nonlocal {H}arnack inequalities for nonlocal heat equations.
\newblock {\em J. Differential Equations}, 267(11):6691--6757, 2019.

\bibitem{Kin-Kuusi}
Juha Kinnunen and Tuomo Kuusi.
\newblock Local behaviour of solutions to doubly nonlinear parabolic equations.
\newblock {\em Math. Ann.}, 337(3):705--728, 2007.

\bibitem{KLin}
Juha Kinnunen and Peter Lindqvist.
\newblock Pointwise behaviour of semicontinuous supersolutions to a quasilinear
  parabolic equation.
\newblock {\em Ann. Mat. Pura Appl. (4)}, 185(3):411--435, 2006.

\bibitem{Kuusiholder2}
Tuomo Kuusi, Rojbin Laleoglu, Juhana Siljander, and Jos\'{e}~Miguel Urbano.
\newblock H\"{o}lder continuity for {T}rudinger's equation in measure spaces.
\newblock {\em Calc. Var. Partial Differential Equations}, 45(1-2):193--229,
  2012.

\bibitem{Kuusiholder}
Tuomo Kuusi, Juhana Siljander, and Jos\'{e}~Miguel Urbano.
\newblock Local {H}\"{o}lder continuity for doubly nonlinear parabolic
  equations.
\newblock {\em Indiana Univ. Math. J.}, 61(1):399--430, 2012.

\bibitem{Nonlocalharnackfirst}
N.~S. Landkof.
\newblock {\em Foundations of modern potential theory}.
\newblock Springer-Verlag, New York-Heidelberg, 1972.
\newblock Translated from the Russian by A. P. Doohovskoy, Die Grundlehren der
  mathematischen Wissenschaften, Band 180.

\bibitem{Liao}
Naian {Liao}.
\newblock {Remarks on parabolic De Giorgi classes}.
\newblock {\em arXiv e-prints}, page arXiv:2004.14324, April 2020.

\bibitem{Rossi}
Jos\'{e}~M. Maz\'{o}n, Julio~D. Rossi, and Juli\'{a}n Toledo.
\newblock Fractional {$p$}-{L}aplacian evolution equations.
\newblock {\em J. Math. Pures Appl. (9)}, 105(6):810--844, 2016.

\bibitem{Martinharnack}
Martin Str\"{o}mqvist.
\newblock Harnack's inequality for parabolic nonlocal equations.
\newblock {\em Ann. Inst. H. Poincar\'{e} Anal. Non Lin\'{e}aire},
  36(6):1709--1745, 2019.

\bibitem{Martinlocal}
Martin Str\"{o}mqvist.
\newblock Local boundedness of solutions to non-local parabolic equations
  modeled on the fractional {$p$}-{L}aplacian.
\newblock {\em J. Differential Equations}, 266(12):7948--7979, 2019.

\bibitem{Vazquez1}
Juan~Luis V\'{a}zquez.
\newblock The {D}irichlet problem for the fractional {$p$}-{L}aplacian
  evolution equation.
\newblock {\em J. Differential Equations}, 260(7):6038--6056, 2016.

\end{thebibliography}
\end{document}